\documentclass{article}

\usepackage{graphicx}      	
\usepackage{amssymb, amsfonts, amsmath, dsfont}
\usepackage{dsfont,hyperref}
\usepackage[normalem]{ulem}
\usepackage{color,bbm}  	
\usepackage{soul}
\usepackage{comment}
\graphicspath{{Images/}}
\newcommand{\A}{\mathcal{A}}

\newcommand{\G}{\mathcal{G}}
\newcommand{\E}{\mathcal{E}}

\newcommand{\ov}{\overline}

\newcommand{\R}{\mathcal{R}}

\newcommand{\V}{\mathcal{V}}

\newcommand{\RR}{\mathbb{R}}

\newcommand{\sgn}[1]{\mathrm{sgn}\left(#1\right)}
\DeclareMathOperator*{\argmax}{argmax}
\DeclareMathOperator*{\argmin}{argmin}
\renewcommand*{\mid}{\mathrm{mid}}
\newcommand{\be}{\begin{equation}}
\newcommand{\ee}{\end{equation}}
\newcommand{\mc}{\mathcal}
\newcommand{\ds}{\displaystyle}
\newcommand{\ba}{\begin{array}}\newcommand{\ea}{\end{array}}
\newcommand{\de}{\mathrm{d}}
\newcommand{\margindef}[1]{}

\newtheorem{proposition}{Proposition}

\newtheorem{lemma}{Lemma}
\newtheorem{remark}{Remark}

\newtheorem{proof}{Proof}
\title{\LARGE \bf
Convexity and Robustness\\ of Dynamic Traffic Assignment\\ and Freeway Network Control}

\author{Giacomo Como
\thanks{G.~Como is with the Department of Automatic Control, Lund University, Sweden. {\tt\small giacomo.como@control.lth.se}. {E.~Lovisari was with Univ.~J.~Fourier and GIPSA-lab, CNRS, Grenoble, France. {\tt\small {enrico.lovisari@gipsa-lab.fr}}}. K.~Savla is with the Sonny Astani Department of Civil and Environmental Engineering at the University of Southern California, Los Angeles, CA. {\tt\small ksavla@usc.edu}. The authors are listed in alphabetical order.}
\and Enrico Lovisari
\and
Ketan Savla
}

\begin{document}
\maketitle
%

%
%
%
%
%
%
\begin{abstract}                
We study the use of the System Optimum (SO) Dynamic Traffic Assignment (DTA) problem to design optimal traffic flow controls for freeway networks as modeled by the Cell Transmission Model, using variable speed limit, ramp metering, and routing. We consider two optimal control problems: the DTA problem, where turning ratios are part of the control inputs, and the Freeway Network Control (FNC), where turning ratios are instead assigned exogenous parameters. It is known that relaxation of the supply and demand constraints in the cell-based formulations of the DTA problem results in a linear program. However, solutions to the relaxed problem can be infeasible with respect to traffic dynamics. Previous work has shown that such solutions can be made feasible by proper choice of ramp metering and variable speed limit control for specific traffic networks. We extend this procedure to arbitrary networks and provide insight into the structure and robustness of the proposed optimal controllers. For a network consisting only of ordinary, merge, and diverge junctions, where the cells have linear demand functions and affine supply functions with identical slopes, and the cost is the total traffic volume, we show, using the maximum principle, that variable speed limits are not needed in order to achieve optimality in the FNC problem, and ramp metering is sufficient. We also prove bounds on perturbation of the controlled system trajectory in terms of  perturbations in initial traffic volume and exogenous inflows. These bounds, which leverage monotonicity properties of the controlled trajectory, are shown to be in close agreement with numerical simulation results.
\end{abstract}

\section{Introduction}
\label{section:Introduction}

The System Optimum (SO) Dynamic Traffic Assignment (DTA) problem, introduced in \cite{MerchantTSa:78, MerchantTSb:78}, has attracted significant interest from the transportation research community, see, e.g., \cite{PeetaNSE:01} for an overview. While originally proposed mainly for planning purposes, it is also being increasingly used as a framework to compute optimal control for traffic flow over freeway networks, e.g., see \cite{GomesTRC06, MuralidharanACC12}, when traffic controllers aim to minimize a global cost of the whole network -- hence Social Optimality, as opposed to the single-vehicle oriented User Equilibrium modeling frameworks. Continuing along these relatively recent trends, this paper focuses on the use of solutions of two variants of the SO-DTA to design optimal controls for dynamic network traffic flows over a given time horizon, in the form of variable speed limit, ramp metering, and routing (turning ratios) matrices. 

The Cell Transmission Model (CTM), originally proposed in \cite{Daganzo:94,Daganzo:95}, is a compelling framework to simulate realistic first order traffic dynamics. It consists of a time and space discretization of the kinematic wave models of Lighthill-Whitham and Richards \cite{LighthillTrafficPTRS55,Richards:56}. Unfortunately, straightforward formulations of DTA for the CTM are known to lead to non-convex problems, and hence are unsuitable especially for real-time applications. On the other hand, in the DTA formulation of \cite{Ziliaskopoulos:00}, the supply and demand constraints of the CTM are relaxed to yield a linear program. However, the computational simplicity resulting from this relaxation comes at the expense of possible infeasibility of a resulting optimal solution with respect to traffic dynamics. Quite interestingly, \cite{MuralidharanACC12} shows that the optimal solution of a linear program analogous to the DTA relaxation in \cite{Ziliaskopoulos:00} can be realized exactly for traffic dynamics modeled by the link-node cell transmission model by proper design of ramp metering and variable speed limit controller, when demand functions are linear, supply functions are affine, and the network consists of a mainline with on- and off-ramps. In this paper, we consider extensions of the approach proposed in \cite{GomesTRC06,MuralidharanACC12} to arbitrary networks, where traffic dynamics is inherited by the CTM, with traffic dynamics encompassing the ones originally proposed in \cite{Daganzo:94,Daganzo:95}, and also allowing for arbitrary concave demand and supply functions, and convex cost functions, including total travel time, total travel distance, and total delay as special cases. Under these generalizations, the resulting S0-DTA is a convex program, which can be solved using readily available software tools such as \texttt{cvx} \cite{CVXSoftware,GrantRALC08} and possibly suited for distributed iterative solvers, e.g., see our preliminary work \cite{Ba.Savla.ea:CDC15}.  

In short, the considered approach to the optimal control of freeway networks over a finite time horizon consists of two steps: (i) to formulate and solve convex optimal control problems that are relaxations either of the DTA problem (with turning ratios part of the optimization), or of the Freeway Network Control (FNC) problem (where the turning ratios are exogenously imposed); (ii) to design open-loop variable speed limits, ramp meters and routing controls over the time horizon to make such optimal solution feasible with respect to traf w dynamics modeled by the CTM. Natural questions concern: (a) under which conditions step (ii) above is not necessary, i.e., the optimal solution of the convex optimal control problem is readily feasible with respect to traffic flow dynamics; and (b) how robust the optimal control computed through the procedure above is with respect to perturbations to the network during the execution of the open-loop controller over the time horizon. The main novel contributions of this paper address questions (a) and (b) as follows. 
On the one hand, using Pontryagin's maximum principle, we prove that, for networks consisting only of ordinary, merge, and diverge junctions, and whose cells have linear demand and affine supply functions with identical slopes, the optimal solution of the FNC problem with total traffic volume as cost does not require the use of variable speed limits as a proper choice of ramp metering controls makes it readily feasible with respect to the CTM model of \cite{Daganzo:94}. 
On the other hand, we provide bounds on the perturbation to the system trajectory under the open-loop controller obtained from the two-step procedure due to perturbations in the initial traffic volumes and exogenous inflows. In order to derive such bounds, we use the fact the nominal (i.e., unperturbed) controlled system trajectory resulting from steps (i) and (ii) above is always in free-flow, and hence it satisfies a certain monotonicity property that in turn can be used to evaluate its robustness to perturbations of the initial traffic volumes and exogenous inflows. 

It is helpful to clarify the novelty of our contributions with respect to existing literature. We generalize the applicability of the two-step procedure of using solution of the two DTA variants for design of optimal traffic flow control to general network topologies, concave
supply and demand functions, convex cost functions, than the ones considered previously, e.g., in \cite{GomesTRC06, MuralidharanACC12}. The maximum principle has been used to identify necessary conditions for optimal control of traffic flow over networks, e.g., in \cite{Friesz.Luque.ea:89}. However, the underlying model for traffic flow dynamics in \cite{Friesz.Luque.ea:89} does not capture backward propagation of congestion, and in particular does not resemble CTM. 

Robustness of open-loop controllers can be quantified through standard sensitivity analysis of controlled traffic dynamics. Our bounds, which exploit the monotonicity properties of controlled system trajectories, are applicable to relatively larger perturbations than those obtained through such standard techniques.
Implications of such monotonicity property for robustness of dynamic network flows have been recently investigated in different contexts \cite{ComoPartITAC13,ComoPartIITAC13,Como.Lovisari.ea:TCNS15}. 
Our robustness analysis of the solution to deterministic DTA problems is to be contrasted with chance-constrained solution of stochastic SO-DTA, e.g.,  in \cite{Waller.Ziliaskopoulos:06} under probabilistic information about inflows. Our approach to applying maximum principle and robustness bounds necessitates consideration of continuous time versions of CTM in the analytical part of the paper. Such continuous time and discrete space versions have been used previously, e.g., to develop probabilistic versions of CTM~ \cite{Jabari.Liu:12}.
Our adoption of continuous time version is merely to facilitate analysis, and is not to be interpreted as a new numerical framework for traffic flow dynamics.
Indeed, the simulations reported in this paper are performed in the standard  discrete time discrete space version of CTM.  

The remained of the paper is organized as follows. In Section \ref{section:traffic_model} the DTA and FNC are formulated as optimal control problems for continuous-time cell-based dynamic traffic models.  In Section~\ref{section:DTA_problems}, convex optimal control relaxations of the DTA and FNC problems are presented and proved to be tight through a proper choice of open-loop variable speed, ramp metering, and routing controllers (Proposition \ref{proposition:implementation}). 
In Section \ref{sec:optimal-control-necessary-condition}, scenarios under which no variable speed and ramp metering control is required in order  to achieve optimality are identified (Proposition \ref{prop:optimal-control-FNC}). Section~\ref{section:robustness} presents perturbations bounds in optimal DTA and FNC solutions due to perturbations in initial traffic volumes and exogenous inflows: the main results are stated in Propositions \ref{theo:sensitivity} and \ref{thm:robustness}. Section \ref{sec:simulations-implementation} reports some numerical simulations illustrating the main theoretical results. Finally, Section~\ref{section:conclusion} draws conclusions and suggests future research directions. The proofs of the technical results are presented in the Appendix.

\section{Dynamic Traffic Model and Problem Formulation}
\label{section:traffic_model}


We describe the topology of the transportation network as a directed multi-graph $\mc G=(\mc V,\mc E)$ with links $i\in\mc E$ representing cells and nodes representing either junctions (being them of merge, diverge, or mixed type) or interfaces between consecutive cells (briefly referred to as ordinary junctions). The head and tail nodes of a cell $i$ are denoted by $\tau_i$ and $\sigma_i$, respectively, so that the cell $i=(\sigma_i,\tau_i)$ is directed from $\sigma_i$ to $\tau_i$. One particular node $w\in\mc V$ represents the external world, with cells $i$ such that $\sigma_i=w$ representing sources (identifiable with on-ramps in freeway networks) and cells $i$ such that $\tau_i=w$ representing sinks (representing off-ramps in freeway networks). The sets of sources and sinks are denoted by $\mc R$ and $\mc S$, respectively. The network topology is typically illustrated by omitting such external node $w$ and letting sources have no tail node and sinks have no head node. (See Figure \ref{fig:ExampleNetwork}.) 
\begin{figure}
\centering
\includegraphics[scale=.8]{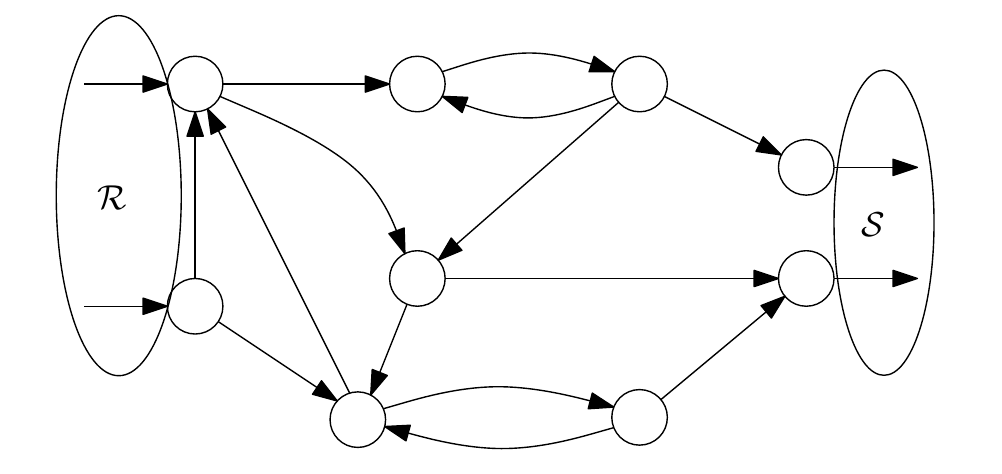}
\caption{A multi-source multi-sink transportation network topology.}
\label{fig:ExampleNetwork}
\end{figure}
The notation 
	$$
		\A 
		= 
		\{(i,j)\in\E\times\E:\,\tau_i=\sigma_j\ne w\}
	$$ 
is used to refer to the set of all pairs of adjacent (consecutive) cells.\footnote{Each cell is meant to model a portion of an actual road of a given length, so that the graph $\G$ need not in general coincide with the actual road network topology. In fact $\G$ may represent a refinement of such actual road network topology, with a given road possibly split into multiple adjacent cells. Observe that such refinement maintains the original junction layout as well as the original sets of sources and sinks.} 

The dynamic state of the network is described by a time-varying vector $x(t)\in\RR^{\mc E}$ whose entries $x_i(t)$ represent the traffic volume in the cells $i\in\E$ at time $t\ge0$. The inputs to the network are the exogenous inflows $\lambda_i(t)\ge0$ at the sources $i\in\R$, while the outputs are the outflows $\mu_i(t)\ge0$ at the sinks $i\in\mc S$. Conventionally, we set $\lambda_i(t)\equiv0$ for all non-source cells $i\in\E\setminus\R$ and $\mu_i(t)\equiv0$ for all non-sink cells $i\in\E\setminus\mc S$, and stack up all the inflows in a vector $\lambda(t)\in\RR^{\E}$ and all the outflows in a vector $\mu(t)\in\RR^{\E}$. We denote the traffic flow from a cell $i$ to an adjacent cell $j$ by $f_{ij}(t)$, the total inflow in and the total outflow from a cell $i$ by $y_i(t)$ and $z_i(t)$, respectively, and stack up all cell-to-cell flows, total inflows, and total outflows in vectors $f(t)\in\RR^{\mc A}$, $y(t)\in\RR^{\mc E}$, and $z(t)\in\RR^{\mc E}$, respectively. The law of mass conservation then reads 
	\be
		\label{dynamics}
		\dot x_i(t)=y_i(t)-z_i(t)\,,\qquad i\in\mc E\,,
	\ee
	\be
		\label{flow-const}
		y_i(t)=\lambda_i(t)+\sum_{j\in\mc E}f_{ji}(t)\,,\qquad z_i(t)=\mu_i(t)+\sum_{j\in\mc E}f_{ij}(t)\,,\qquad i\in\mc E\,,
	\ee
In fact, equation \eqref{dynamics} states that the time-derivative of the traffic volume $x_i$ on a cell equals the imbalance between its inflow $y_i$ and its outflow $z_i$, while equation \eqref{flow-const} states that the inflow $y_i$ in a cell is the aggregate of the exogenous inflow $\lambda_i$ and the flows $f_{ji}$ from other cells in the network, while, symmetrically, the outflow $z_i$ from a cell is the aggregate of the flows $f_{ij}$ to other cells in the network and the outflow $\mu_i$ towards the external world.\footnote{In fact, one could have replaced $y_i$ and $z_i$ in both the right-hand sides of \eqref{dynamics} and \eqref{supply-demand-const} by the expressions in the right-hand sides in \eqref{flow-const} and reduced the number of variables in the DTA problem. However, introducing the variable $y_i$ and $z_i$ turns out to be useful for the problems discussed in Section~\ref{section:DTA_problems}. Using these additional variables also proves useful in deriving distributed optimization algorithms based, e.g., on the alternating method of multipliers or interior point methods, e.g., see \cite{Ba.Savla.ea:CDC15}.}

\begin{figure}
\centering
\includegraphics[width=.45\textwidth]{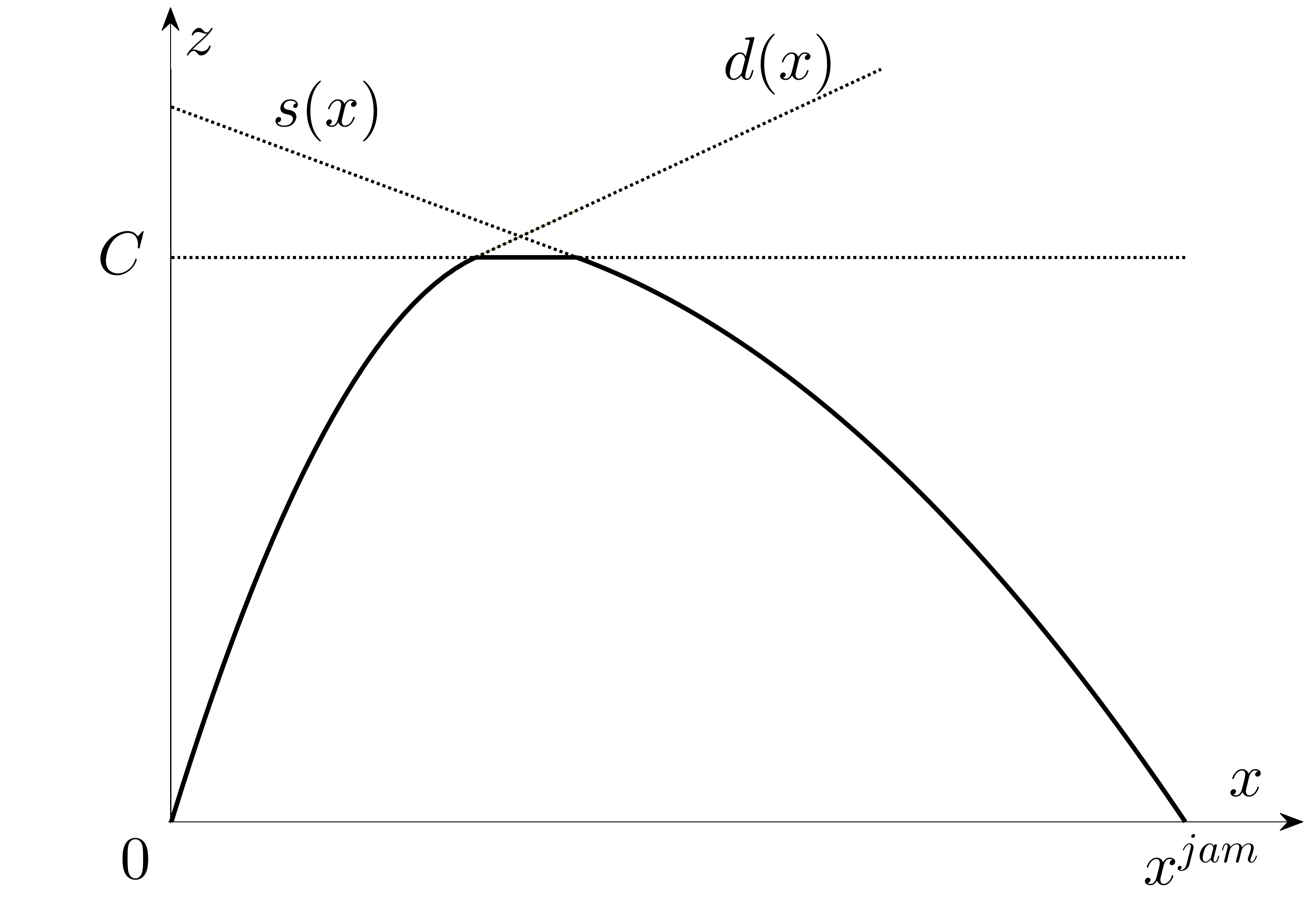}
\includegraphics[width=.45\textwidth]{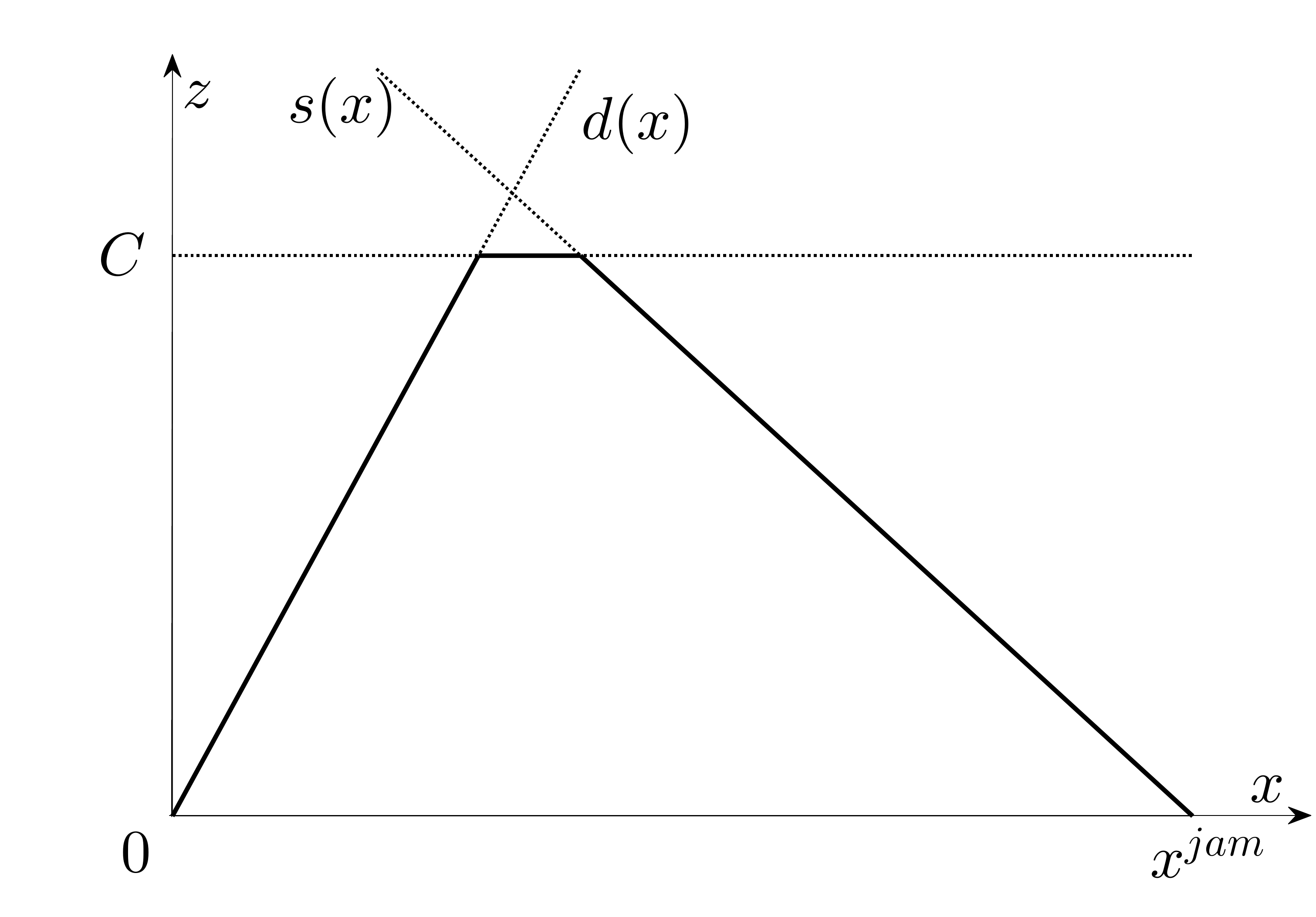}
\caption{Concave (left) and piecewise linear (right)  fundamental diagrams. Observe that the set of pairs $(x,z)$ such that $0\le z\le\min\{d(x),C,s(x)\}$ is a closed convex set (in the special case in the right figure it is in fact a polytope). }
\label{fig:SupplyAndDemandExample}
\end{figure}

We consider first order dynamic traffic models and assume a concave Fundamental Diagram, as the one shown in Figure \ref{fig:SupplyAndDemandExample}, to model the relationship between traffic volume and flows. Following Daganzo's Cell Transmission Model, we introduce demand and supply functions, which return the maximum outflow from and, respectively, the maximum inflow in a cell as a function of its current traffic volume and, possibly, additional control parameters. In turn, they can be interpreted as representing the rising and, respectively, descending parts of the Fundamental Diagram (see again Figure \ref{fig:SupplyAndDemandExample}). 
In particular, the demand functions are assumed to take the form 
\be\label{eq:control-demand-speed}\ov d_i(x_i,\alpha_i)=\min\{\alpha_id_i(x_i),C_i\}\,,\qquad i\in\mc E\setminus\mc R\,,\ee
\be\label{eq:control-demand-rampmetering}\ov d_i(x_i,\alpha_i)=\min\{d_i(x_i),\alpha_i C_i\}\,,\qquad i\in\mc R\,,\ee
where: $d_i(x)$ is a continuous, non-decreasing, and concave function of the traffic volume such that $d_i(0)=0$; 
$C_i=C_i(t)\ge0$ is the possibly time-varying maximum flow capacity; 
$\alpha_i=\alpha_i(t)\in[0,1]$ is a possibly time-varying demand control parameter actuated via speed limit control on the non-source cells $i\in\mc E\setminus\mc R$ and ramp-metering on the sources $i\in\mc R$.\footnote{In the context of freeway networks, \eqref{eq:control-demand-speed}--\eqref{eq:control-demand-rampmetering} can be realized through appropriate setting of speed limits and ramp metering. In particular, for linear uncontrolled demand functions $d_i(x_i)=v_ix_i$, formula \eqref{eq:control-demand-speed} is equivalent to the modulation of the free-flow speed $\ov v_i=v_i\alpha_i$, where $v_i$ could be interpreted as the maximum possible speed due to, e.g., safety considerations (Cf., e.g., \cite{HegyiTRC05}). On the other hand, \eqref{eq:control-demand-rampmetering} corresponds to metering the maximum outflow from the onramp, which is its demand, by imposing a maximum value $c_i=\alpha_iC_i$ (Cf., e.g., \cite{MuralidharanACC12}).} (See Figure \ref{fig:DemandControl}.)
On the other hand, the supply functions $s_i(x_i)$ of every non-source cell $i\in\mc E\setminus\mc R$ are assumed to be continuous, non-increasing, concave, and such that $s_i(0)>0$, with $x^{\textrm{jam}}_i=\inf\{x_i>0:\,s_i(x_i)=0\}$ denoting cell $i$'s jam traffic volume. Conventionally, $s_i(x_i)\equiv+\infty$ at all sources $i\in\mc R$. 

\begin{figure}
\centering
\includegraphics[width=.45\textwidth]{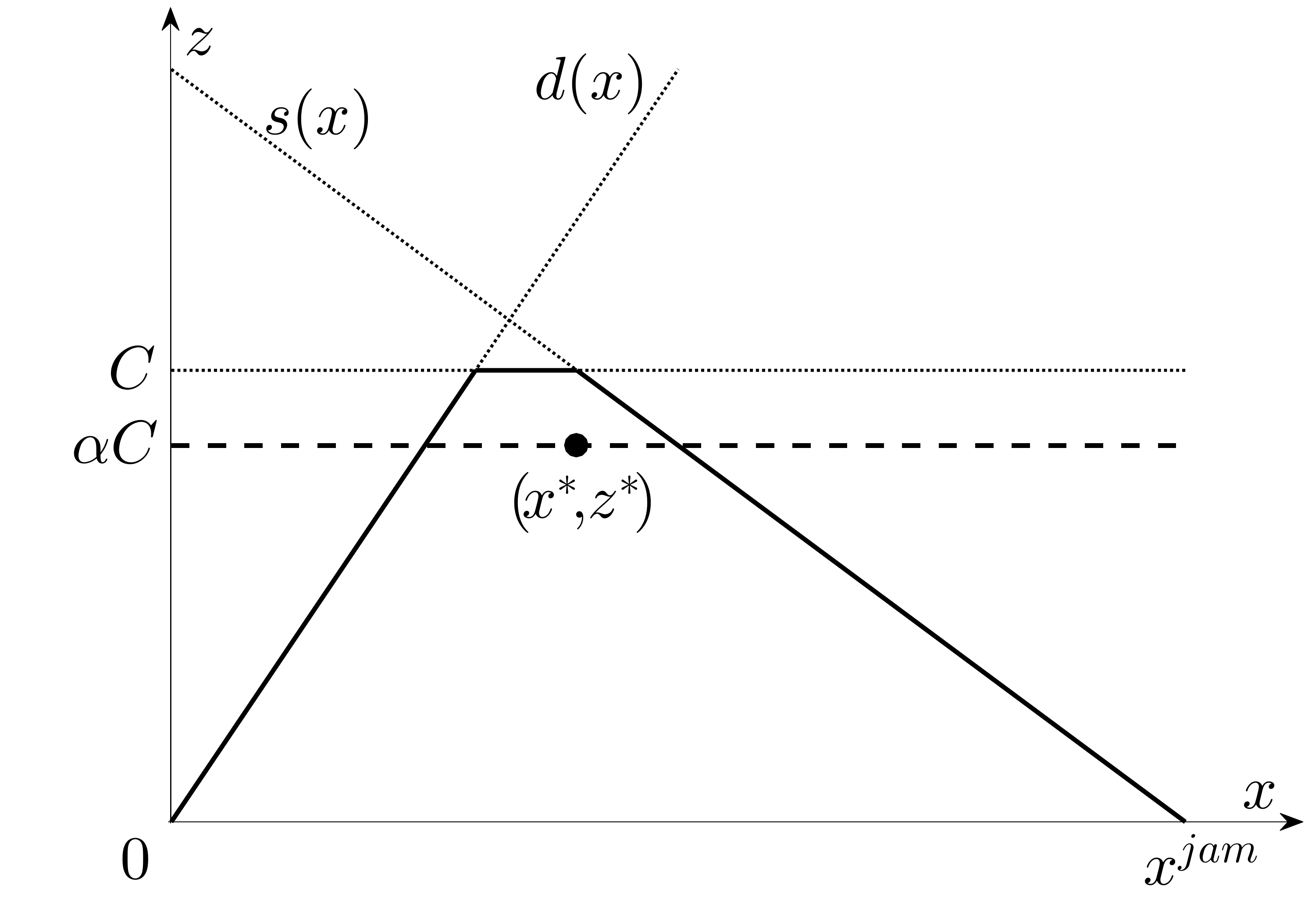}
\includegraphics[width=.45\textwidth]{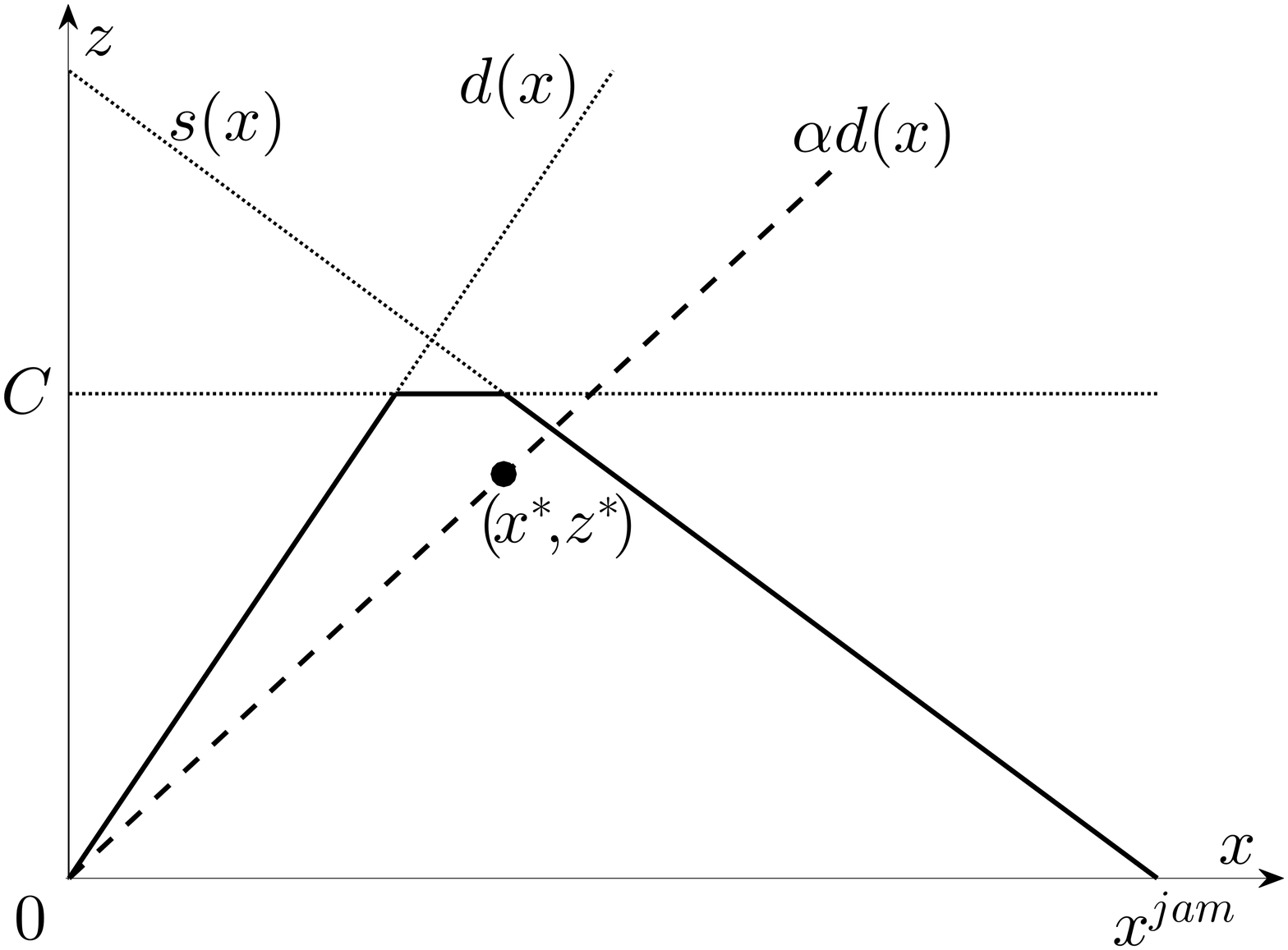}
\caption{Demand control is actuated by ramp metering on the sources $i\in\mc R$ (left figure) and variable speed limit in the other cells $i\in\mc E\setminus\mc R$ (right figure). As proved in Proposition \ref{proposition:implementation}, for any feasible pair $(x^*,z^*)$ such that $0\le z^*\le\min\{d(x^*),C,s(x^*)\}$ there exists a choice of the demand control parameter $\alpha$ such that $z^*=d(\alpha,x^*)$.}
\label{fig:DemandControl}
\end{figure}

Finally, we consider a nonnegative, possibly time-varying, $\mc E\times\mc E$ routing matrix $R=R(t)$ satisfying the network topology constraints 
	\be
		\label{R-const}
		R_{ij}=0\,,\qquad  (i,j)\in(\mc E\times\mc E)\setminus\mc A\,,
	\ee 
and such that 
	\be
		\label{sumR=1}
		\sum_{j\in\mc E}R_{ij}=1\,,\qquad i\in\mc E\setminus\mc S\,.
	\ee
The matrix $R$ is to be interpreted as describing the drivers' route choices, with its entries $R_{ij}$, sometimes referred to as turning ratios, representing the fractions of flow leaving cell $i$ that is directed towards cell $j$. Equation \eqref{sumR=1} then guarantees that all the outflow from the non-sink cells is split among other cells in the network, while equation \eqref{R-const} guarantees that the outflow from cell $i$ is split among adjacent downstream cells only. Depending on the specific problem considered in this paper, the matrix $R$ is to be considered as a control variable, as is the case for the Dynamic Traffic Assignment (DTA) problem, or an exogenous input, as is the case for the Freeway Network Control (FNC) problem.

The traffic dynamics is then described by coupling the mass conservation laws \eqref{dynamics} and \eqref{flow-const} with the following functional dependence of the exogenous outflows and the cell-to-cell flows on the traffic volume and control parameters: 
	\be
		\label{DNL1}
		z_k=\mu_k=\ov d_k(x_k,\alpha_k)\,,\qquad k\in\mc S\,,
	\ee 
\begin{equation}
		\label{eq:fjFIFO}		
		f_{ij} = R_{ij}z_i\,,\qquad z_i=\gamma_i^F\ov d_i(x_i,\alpha_i)\,,\qquad (i,j)\in\mc A\,,
	\end{equation}
where for all $i\in\E$
\begin{equation}
		\label{eq:gammaFIFO}
		\!\!\!\!\!\!\!\gamma^F_i\! = \!\sup\left\{\!\gamma\in[0,1]:\, \gamma\!\cdot\!\!\max_{\substack{k\in\mc E:\\(i,k)\in\mc A}}\ {\sum_{h\in\mc E}R_{hk} \ov d_h(x_h,\alpha_h)}\le{s_k(x_k)}\right\}
	\end{equation}
Equations \eqref{eq:fjFIFO}-\eqref{eq:gammaFIFO} generalize Daganzo's cell-transmission model \cite{Daganzo:95} by extending it to the case where junctions may have multiple incoming and outgoing cells\footnote{Daganzo \cite{Daganzo:95} considers only the cases of nodes with single incoming link (diverge junction) or single outgoing link (merge junction). It is easily seen that, for diverge junctions \eqref{eq:fjFIFO}-\eqref{eq:gammaFIFO} prescribe that the outflow from a diverge junction always splits into the downstream cells according to the turning ratios, and the supply of a congested merge junction is allocated according to a proportional rule. This is a slight variation with respect to the original CTM model \cite{Daganzo:95}, in which merge is solved via a priority rule.}  and the cells' demand and supply functions are allowed to be concave  rather than piecewise linear functions.\footnote{Daganzo's original cell-transmission model \cite{Daganzo:94} assumes a trapezoidal Fundamental Diagram, hence piecewise linear demand and affine supply functions.} 
Since $f_{ij}=R_{ij}z_i$ in every circumstances, equations \eqref{eq:fjFIFO}-\eqref{eq:gammaFIFO} enforce the FIFO (First-In First-Out) constraint and are therefore amenable to the modeling of multi-source multi-sink transportation networks. 
For a given assignment  
	\be \label{initial-const}
		x_i(0)=x_i^0\,,\qquad i\in\mc E\,,
	\ee
of initial traffic volumes in the cells and dynamics of the inflows $\lambda(t)$, routing matrix $R(t)$, and demand control parameters $\alpha(t)$, the evolution of the traffic volume vector $x(t)$ for $t\ge0$ is uniquely determined by equations \eqref{dynamics}, \eqref{flow-const}, and \eqref{DNL1}-\eqref{eq:gammaFIFO}. 
In particular, for every choice of $x^0$, and of $\lambda(t)$, $R(t)$, and $\alpha(t)$ for $t\ge0$, the traffic volumes satisfy $x_i(t)\ge0$ for every cell $i\in\mc E$ (since, when $x_i=0$, then $z_i\le \ov d_i(x_i,\alpha_i)\le d_i(0)=0$) and $x_i(t)\le x_i^{jam}$ for every non-source cell $i\in\mc E\setminus\mc R$ (since, when $x_i=x^{jam}_i$, then $y_i\le s_i(x^{jam}_i)=0$). 

Within this setting, the DTA and the FNC problems can be cast as open-loop optimal control problems consisting in the minimization of the integral of a running cost $\psi(x,z)$ that is a function of vector of traffic volume and outflows over the time interval $[0,T]$, where $T>0$ is a given time horizon. We focus on running costs $\psi(x,z)$ that are convex in $(x,z)$, nondecreasing in each entry $x_i$, nonincreasing in each entry $z_i$, and such that $\psi(0,0)=0$.
A relevant special case is when the cost function is separable, i.e., when 
	\be
		\label{separable-cost}
		\psi(x,z)=\sum_{i\in\mc E}\psi_i(x_i,z_i)\,,
	\ee 
for convex costs $\psi_i(x_i,z_i)$ of the traffic volume and outflow on the single cells $i\in\mc E$, with $\psi_i(0,0)=0$. 
This class of cost functions includes several standard choices \cite{MuralidharanACC12, GomesTRC06} such as
	\begin{itemize}
		\item Total Travel Time, for which $\psi_i(x_i,z_i) = x_i$. In fact,  $\int_0^Tx_i(t)dt$ can be interpreted as the total time spent on cell $i$ by all travelling vehicles in the interval $[0,T]$. 
		\item Total Travel Distance, for which $\psi_i(x_i,z_i) = -\ell_iz_i$, where $\ell_i$ stands for the length of cell $i$. In fact, $\int_0^T\ell_iz_i(t)dt$ is the the distance travelled on cell $i$ by all vehicles that exit from it in the interval $[0,T]$. The minus sign implies that such a cost function is to be maximized, instead of minimized;
		\item Total Delay, for which $\psi_i(x_i, z_i) = x_i - {z_i}/{v_i}$, where $v_i$ is the freeflow speed on cell $i$. In fact, the integral $\int_0^T(x_i(t) - {z_i(t)}/{v_i})\de t$ represents the total additional time spent by vehicles on cell $i$ with respect to the freeflow case $z_i=v_i x_i$ during the interval $[0,T]$.
	\end{itemize}
Furthermore, any linear combination of such costs remains convex and might be used to trade-off between them.	

Given a single-sink transportation network, a running cost $\psi(x,z)$ as above, a finite time-horizon $T>0$, initial cell traffic volumes $x_i^0\ge0$, and inflows $\lambda_i(t)\ge0$ for $0\le t\le T$ at the source cells $i\in\mc R$, the DTA problem for the continuous-time cell-based dynamic traffic model of Section \ref{section:traffic_model} can be formulated as follows: 
\be\label{DTA-0}
\min_{\substack{\alpha(t),R(t):\\\eqref{dynamics},\eqref{flow-const},\eqref{R-const}-\eqref{initial-const}
}}\int_0^T\psi(x(t),z(t))\de t\,.
\ee
Similarly, the FNC problem in a possibly multi-sink network can be formulated as follows, for a given running cost $\psi(x,z)$, finite time-horizon $T>0$, initial cell traffic volumes $x_i^0\ge0$, exogenous inflows $\lambda_i(t)\ge0$ at the sources $i\in\mc R$ and a routing matrix $R(t)$ satisfying \eqref{R-const}-\eqref{sumR=1} for $0\le t\le T$:
\be\label{FNC-0}
\min_{\substack{\alpha(t):\\\eqref{dynamics},\eqref{flow-const},\eqref{DNL1}-\eqref{initial-const}
}}\int_0^T\psi(x(t),z(t))\de t\,.
\ee

Observe that the optimal control formulations \eqref{DTA-0} for the DTA and \eqref{FNC-0} of the FNC problem differ because of the role of the routing matrix that is an endogenous optimization variable in the former, while it is exogenously determined in the latter. Both \eqref{DTA-0} and \eqref{FNC-0} are non-convex optimal control problems, which are hard to be either analyzed or numerically solved. The source of non-convexity are the equations \eqref{DNL1}-\eqref{eq:gammaFIFO} which are nonlinear equality constraints. In the next section, tight convex relaxation of the optimal control problems \eqref{DTA-0} and \eqref{FNC-0} will be discussed.

\section{Tight Convexifications of the DTA and the FNC problems}
\label{section:DTA_problems}

In this section, we show how, in spite of their non-convexity, the DTA and the FNC problems \eqref{DTA-0} and \eqref{FNC-0} introduced in Section \ref{section:traffic_model} are amenable to be suitably reparametrized and relaxed in such a way that the obtained relaxations are both convex and tight, i.e., their optimal solution can be mapped back into a solution of the original optimal control problems with a proper choice of the demand control parameters (i.e., ramp metering and variable speed limits) and, for the DTA problem only, of the routing matrix. As will be clarified later, the approach is a generalization of the ones proposed in \cite{Ziliaskopoulos:00} and \cite{GomesTRC06}.

We start by presenting below two optimal control problems that are convex, and are thus amenable to be analyzed and solved in a computationally efficient way. We then prove how an optimal solution of such convex problems can be mapped into one of the original problems  \eqref{DTA-0} and \eqref{FNC-0}. The convex optimal control problems that we introduce do not directly involve the control variables $\alpha(t)$ and $R(t)$ as \eqref{DTA-0} and \eqref{FNC-0}, but rather the flow variables $f(t)$, $y(t)$, and $z(t)$ along with the traffic volume vector $x(t)$. Besides the law of conservation of mass \eqref{dynamics} and \eqref{flow-const} and the initial traffic volume equation \eqref{initial-const}, such variables are required to satisfy the nonnegativity constraints
	\be
		\label{topology-const}
			f_{ij}(t)\ge0\,,\quad(i,j)\in\mc A\,,\qquad \mu_i(t)\ge0\,,\quad i\in\mc S\,,\qquad
		 	\mu_i(t)=0\,,\quad i\in\mc E\setminus\mc S\,,
	\ee
and the supply and demand constraints
	\be
		\label{supply-demand-const}
		y_i(t) \le s_i(x_i(t))\,,\qquad z_i(t) \le \min\{d_i(x_i(t)), C_i\}\,,\quad i\in\mc E\,.
	\ee
The inequalities in \eqref{topology-const} enforce non-negativity of the cell-to-cell flows $f_{ij}$ and of the external outflows $\mu_i$\footnote{Observe that, together with \eqref{flow-const} and non-negativity of the exogenous inflows $\lambda_i$, equation \eqref{topology-const} implies non-negativity of the cells' inflows $y_i$ and outflows $z_i$ as well.} while the equality in \eqref{topology-const} guarantees that outflows $\mu_i$ towards the external world are possible only from the sink nodes.\footnote{Notice that, because of the assumption that $\lambda_i\ne0$ only on source cells, and of the first line of  \eqref{topology-const}, it turns out that at most one between the two inflow terms $\lambda_i$ and $\sum_jf_{ji}$ appearing in the righthand side of the first equation in \eqref{flow-const} can be positive: the exogenous inflow $\lambda_i$ for onramps $i\in\mc R$, and the aggregate inflow from other cells $\sum_jf_{ji}$ for every $i\in\mc E\setminus\mc R$. Similarly, \eqref{topology-const} implies that only one of the outflow terms $\mu_i$ and $\sum_jf_{ij}$ appearing in the righthand side of the second equation in \eqref{flow-const} can be positive: the external outflow $\mu_i$ for off ramps $i\in\mc S$, and the aggregate outflow towards other cells $\sum_jf_{ij}$ for all $i\in\mc E\setminus\mc S$.} 
On the other hand, the inequalities in \eqref{supply-demand-const} capture the physical constraints on the cells: they guarantee that the total inflow $y_i$ in a cell does not exceed the supply $s_i(x_i)$, and the total outflow $z_i$ from a cell does not exceed the demand $\min\{d_i(x_i), C_i\}$. Because of the assumption $d_i(0)=0$ and non-negativity of the cell inflow $y_i$, equation \eqref{dynamics} implies that the traffic volume $x_i(t)$ remains nonnegative in time on every cell $i\in\mc E$. Analogously, for non-sink cells $i\in\mc E\setminus\mc S$, non-negativity of the outflow $z_i$ and the fact that $s_i(x_i^{jam})=0$ imply that the traffic volume $x_i(t)$ never exceeds the jam volume $x^{jam}_i$.


The first convex optimal control problem that we consider is a relaxation of the DTA \eqref{DTA-0} that reads 
	\be
		\label{DTA-1}
\min_{\substack{x,y,z,f,\mu:\\\eqref{dynamics},\eqref{flow-const},\eqref{initial-const},\eqref{topology-const},\eqref{supply-demand-const}}}
\int_0^{T}\psi(x(t),z(t))\de t\,.
	\ee
By including an additional constraint ensuring that the outflow from cell $i$ be split exactly as prescribed by an exogenous routing matrix $R(t)$, i.e., 
	\be
		\label{additional-const-2} 
		f_{ij}(t)= R_{ij}(t)z_i(t)\,,\qquad i\in\mc E\,,
	\ee
we obtain a second convex optimal control problem 
	\be
		\label{FNC-1}
\min_{\substack{x,y,z,f,\mu:\\\eqref{dynamics},\eqref{flow-const},\eqref{initial-const},\eqref{topology-const},\eqref{supply-demand-const},\eqref{additional-const-2}}}
\int_0^{T}\psi(x(t),z(t))\de t\,.
	\ee
that is a relaxation of the FNC problem \eqref{FNC-0}. 

Convexity of the optimal control problems \eqref{DTA-1} and \eqref{FNC-1} means that, if $(x^{(0)}(t),\!y^{(0)}(t),z^{(0)}(t),f^{(0)}(t),\!\mu^{(0)}(t))$ and $(x^{(1)}(t),\!y^{(1)}(t),z^{(1)}(t),f^{(1)}(t),\!\mu^{(1)}(t))$ both satisfy the constraints in  \eqref{DTA-1} (respectively, in \eqref{FNC-1}), then, for every $\beta$ in $[0,1]$, also $(x^{(\beta)},y^{(\beta)},z^{(\beta)},f^{(\beta)},\mu^{(\beta)})$ does, where $x^{(\beta)}=(1-\beta)x^{(0)}+\beta x^{(1)}$, $y^{(\beta)}=(1-\beta)y^{(0)}+\beta y^{(1)}$  and so on,  and 
\be\label{convex-cost}\int_0^T\psi(x^{(\beta)}(t))\de t\le(1-\beta)\int_0^T\psi(x^{(0)}(t))\de t+\beta\int_0^T\psi(x^{(1)}(t))\de t\,.\ee
In fact, it can be easily verified that concavity of the supply and demand functions implies that \eqref{supply-demand-const} is a convex inequality constraint, while the remaining constraints \eqref{dynamics},\eqref{flow-const},\eqref{initial-const},\eqref{topology-const}, and \eqref{additional-const-2} are all linear equalities, and convexity of the cost function $\psi(x,z)$ implies \eqref{convex-cost}. 

On the other hand, the fact that the optimal control problems \eqref{DTA-1} and \eqref{FNC-1} are relaxations of \eqref{DTA-0} and \eqref{FNC-0}, respectively, is easily verifiable. Indeed, for every choice of demand control parameters $\alpha(t)$ in $[0,1]^{\mc E}$ and routing matrix $R(t)$ satisfying \eqref{R-const} and \eqref{sumR=1}, the traffic dynamics generated by \eqref{dynamics}-\eqref{flow-const} and \eqref{DNL1}-\eqref{initial-const} necessarily satisfies \eqref{topology-const}, \eqref{supply-demand-const}, and \eqref{additional-const-2}, hence it is a feasible solution of the convex optimal control problems \eqref{DTA-1} and \eqref{FNC-1}. 
That these relaxations are tight is implied by the following result, showing that for every feasible ---hence, in particular, for the optimal--- solution of \eqref{DTA-1} (of \eqref{FNC-1}) there exists a choice of demand control parameters $\alpha(t)$ and routing matrix $R(t)$ (a choice of the demand control parameters $\alpha(t)$ for every exogenous routing matrix $R(t)$) such that equations \eqref{DNL1}-\eqref{eq:gammaFIFO} are satisfied, so that the solution is also feasible for \eqref{DTA-0} (respectively, \eqref{FNC-0}).

\begin{proposition}
\label{proposition:implementation}
Let $\mc G=(\mc V,\mc E)$ be a network topology,  $x^0$ a vector of initial traffic volumes, and $\lambda_i(t)$, for $t\ge0$, exogenous inflows at the sources $i\in\mc R$. Then, 
\begin{enumerate}
\item[(i)] For any feasible solution $(x(t),y(t),z(t),\mu(t),f(t))$ of the convex optimal control problem \eqref{DTA-1},  
set the demand controls $\alpha(t)$ and controlled routing matrix $R(t)$, for $t \in [0,T]$, as follows
	\be\label{eq:controlVariablesFC_alpha}
			\alpha_{i}(t)=\left\{\ba{lcl}	\ds{z_i(t)}/{d_i(x_i(t))} & &  i \in \E\setminus\R	\\[7pt] z_i(t)/C_i, &&  i \in \R\,.\ea\right.	\ee		
\be			\label{eq:controlVariablesFC_R}
			\ds R_{ij}(t)= \left\{\ba{lcl}{f_{ij}(t)}/{z_i(t)} && (i,j)\in\mc A\\[7pt]
			0&& (i,j)\in\mc E\times\mc E\setminus\mc A\,,\ea\right.
	\ee
with the convention that $\alpha_i(t) = 1$ if $z_i(t)=d_i(x_i(t)) = 0$ on a non-source cell $i\in\mc E\setminus\mc R$, and that, if $z_i(t) = 0$, then $R_{ij}(t) =|\{k\in\mc E:\,(i,k)\in\mc A\}|^{-1}$ for all $(i,j)\in\mc A$. Then, $R(t)$ satisfies the constraints \eqref{R-const}-\eqref{sumR=1} and $x(t)$ satisfies the controlled dynamics \eqref{DNL1}-\eqref{eq:gammaFIFO}, so that $(\alpha(t),R(t))$ is a feasible solution of the DTA problem \eqref{DTA-0}.
	\end{enumerate}
Moreover, let $R(t)$, $t\in[0,T]$, be an exogenous routing matrix satisfying \eqref{R-const} and \eqref{sumR=1}. Then: 
	\begin{enumerate}
\item[(ii)] For any feasible solution $(x(t),y(t),z(t),\mu(t),f(t))$ of the convex optimal control problem \eqref{FNC-1}, set the demand controls $\alpha(t)$, for $t \in [0,T]$, as in \eqref{eq:controlVariablesFC_alpha}. 
Then, $x(t)$ satisfies the controlled dynamics \eqref{DNL1}-\eqref{eq:gammaFIFO}, so that $\alpha(t)$ is a feasible solution of the FNC problem \eqref{FNC-0}.
%
\end{enumerate}
Furthermore, in both cases (i)--(ii) above, the implemented trajectory is always in free-flow, i.e., $z_i(t)=\ov d_i(x_i(t),\alpha_i(t))$ for all $t\in[0,T]$.
\end{proposition}
The proof of Proposition \ref{proposition:implementation} is provided in Appendix \ref{proof:proposition:implementation}. See Figure \ref{fig:DemandControl} for a graphical interpretation of the chosen demand controls $\alpha_i(t)$ on the source and on the non-course cells, respectively. 

Proposition~\ref{proposition:implementation} provides a methodology to take any feasible solution of the convex optimal control problem \eqref{DTA-1} (respectively, \eqref{FNC-1}), and make it feasible with respect to the DTA problem \eqref{DTA-0} (respectively, to the FNC problem \eqref{FNC-0}). 
The convex optimal control problem \eqref{DTA-1} is a continuous-time version of the DTA formulation considered by \cite{Ziliaskopoulos:00} in discrete time and in the special case when the cost			
	\be
		\label{cost=TTT}
		\psi(x)=\sum_{i\in\mc E}x_i
	\ee 
is the total traffic volume in the network and the demand functions $d_i(x_i)$ and supply functions  $s_i(x_i)$ are piecewise affine. In this case ---and more in general, for linear, not necessarily identical, cost functions--- \eqref{DTA-1} is a linear program. In our more general formulation, where the cost $\psi(x)$ is allowed to be a convex function of the traffic volume vector $x$ and the demand and supply functions are concave, \eqref{DTA-1} is a convex (infinite-dimensional) program. 
The special case \eqref{separable-cost} where the cost is separable and convex is suitable for efficient solutions, e.g., based on distributed iterative algorithms \cite{Ba.Savla.ea:CDC15}.
On the other hand, Proposition \ref{proposition:implementation} generalizes existing results, e.g., see \cite{MuralidharanACC12}, which are applicable only in specific scenarios, to convex costs, concave demand and supply functions, and arbitary network topologies.  

\section{Analysis of the FNC problem}
\label{sec:optimal-control-necessary-condition}


In this section, we apply optimal control techniques to the study of the convex formulation of the FNC problem \eqref{FNC-1}. Throughout, we focus on the case when the running cost $\psi(x)$ is a function of the traffic volume vector only.  As a first step, it proves convenient to use the outflows $z_i$ as the only control variables  and substitute $f_{ij}=R_{ij}z_i$ for $(i,j)\in\mc A$, $\mu_k=z_k$ for $k\in\mc S$, and $y_i=\lambda_i+\sum_{j}R_{ji}z_j$ for $i\in\mc E$. 
Then, the convex optimal control problem \eqref{FNC-1} can be reformulated as 
	\be\label{FNC}
\ba{c}\ds\min\int_0^{T}\psi(x(t))\de t\\[8pt] 
\ds x_i(0)=x_i^0\,,\qquad\ds\dot x_i(t)=\lambda_i(t)+\sum_{j\in\mc E}R_{ji}(t)z_j(t)-z_i(t)\,,\qquad i\in\mc E\,,\\[8pt]
		\ds 0\le z_i(t)\le\min\{d_i(x_i(t)),C_i\}\,,\qquad\sum_{j\in\mc E}R_{ji}(t)z_j(t)\le s_i(x_i(t))\,,\qquad i\in\mc E\,.\ea
	\ee
The Hamiltonian associated to the optimal control problem \eqref{FNC} is given by 
	\be
		H(x,z,\zeta)= \psi(x) + \sum_{i\in\mc E} \zeta_i \Big(\lambda_i+\sum_{j\in\mc E}R_{ji} z_j - z_i\Big)\,,
	\ee
where $x$ is the state vector, $z$ is the control vector, and $\zeta$ is the adjoint state vector. Upon introducing the notation
	$$
		\kappa_i(t)=\zeta_i(t)-\sum_{j\in\mc E}R_{ij}\zeta_j(t)\,,\qquad i\in\mc E\,,
	$$
the Hamiltonian can be rewritten as 
	\be
		\label{eq:hamiltonian-FC} 
		H(x,z,\zeta)= \psi(x) + \sum_{i\in\mc E} \zeta_i \lambda_i-\sum_{i\in\mc E}\kappa_iz_i\,. 
	\ee
Then, the Pontryagin maximum principle implies the following necessary condition: if $(x^*(t),z^*(t))$ is an optimal solution of \eqref{FNC}, then, for every $t\in[0,T]$,
	\be
		\label{LP}
		z^*(t)\in\argmin_{\substack{\\[0pt]\ds (z_i)_{i\in\E}\\[5pt]\ds\ds0\le z_i\le C_i\\[5pt]\ds z_i\le d_i(x^*_i(t))\\[5pt]\ds \sum_{j\in\mc E}R_{ji}z_j\le s_i(x^*_i(t))}} 
		H(x^*(t),z,\zeta(t))=\argmax_{\substack{\\[0pt]\ds (z_i)_{i\in\E}\\[5pt]\ds 0\le z_i\le C_i\\[5pt]\ds 0\le z_i\le d_i(x_i^*)\\[5pt]\ds \sum_{j\in\mc E}R_{ji}z_j\le s_i(x_i^*)}} \sum_{i\in\mc E}\kappa_i(t)z_i\,.
		\ee
For a given value of the optimal traffic volume vector $x^*(t)$ and of the adjoint state vector $\zeta(t)$ (hence of $\kappa(t)$), the optimization in rightmost side of \eqref{LP} is a linear problem in the variables $z_i$. 
If one denotes by $\xi_i$, $\nu_i$, and $\chi_i$ the multipliers associated to, respectively, the demand, the supply, and the capacity constraints, then the dual problem to one in the rightmost side of \eqref{LP} can be written as

	\be
		\label{LP-dual}
		(\xi^*(t),\nu^*(t),\chi^*(t))\in
		\!\!\!\!\!\!\!\!
		\argmin_{\substack{\\[0pt]\ds (\xi_i, \nu_i)_{i\in\E}\\[5pt]\ds \xi_i\ge0,\ \nu_i\ge0,\ \chi_i\ge0\\[5pt]\ds \xi_i+\chi_i+\sum_{j\in\mc E}R_{ij}\nu_j\ge\kappa_i}}
		\!\!\!\!\!\!\!\!
		\sum_{i\in\mc E}\left(\xi_id_i(x^*_i(t))+\nu_is_i(x^*_i(t))+\chi_iC_i\right)\,,
	\ee
and the adjoint dynamical equations read 
	\be
		\label{taudot}
		\dot\zeta_i(t)=-\frac{\partial}{\partial x_i}\psi(x^*(t))+\xi_i^*(t)d_i'(x_i^*(t))+\nu_i^*(t)s_i'(x^*_i(t))\,,\qquad i\in\mc E\,,
	\ee
where $d_i'(x_i)$ and $s_i'(x_i)$ are the derivatives of the demand and, respectively, of the supply functions, with transversality condition
	\be
		\label{tauT}
		\zeta_i(T)=0\,,\qquad i\in\mc E\,.
	\ee
If every node $v$ in the network is either an ordinary junction (single incoming and outgoing cell), a merge junction (multiple incoming and single outgoing cells), or a diverge junction (single incoming and multiple outgoing cells), then it is convenient to regroup the addends of the summation appearing in the rightmost side of \eqref{LP} so that the necessary condition for optimality can be separated into decoupled  local linear programs
	\be
		\label{LP-bis} 
		(z_i^*(t))_{i:\tau_i=v}\in\argmax_{\substack{\\[0pt]\ds (z_i)_{i:\tau_i=v}\\[5pt]\ds0\le z_i\le C_i\\[5pt]\ds  z_i\le d_i(x_i^*(t))\\[5pt]\ds \sum_{i:\,\tau_i=v}R_{ij}z_i\le s_j(x_j^*(t)),\,\,\,\forall j\!:\sigma_j=v}}\sum_{i:\tau_i=v}\kappa_i(t)z_i\,,\qquad \forall v\in\mc V\,.
	\ee
	
In the special case where the cost is the total traffic volume, the increasing parts of the demand functions are linear and the supply functions are affine, all with identical slopes, equations \eqref{taudot}---\eqref{LP-bis} imply the following result on the structure of the optimal control. 

\begin{proposition}
\label{prop:optimal-control-FNC}
Let the cost function be as in \eqref{cost=TTT}, the increasing parts of the demand functions be $d_i(x_i)=\omega x_i$, and the supply functions have the form $s_i(x_i)=\theta_i-\omega x_i$ for some $\theta_i>0$ and $\omega >0$. Let $R$ be a routing matrix that is constant in time, and let $(x^*(t),z^*(t))$ be an optimal solution of the corresponding FNC. 
Then, for every non-sink cell $i\in\mc E\setminus\mc R_o$,  
\begin{enumerate}
\item[(i)]
if $v=\tau_i$ is an ordinary junction with downstream cell $j$ (with $\sigma_j=v$), then 
\be\label{ordinary} z_i^*(t)=\min\{d_i(x_i^*(t)),s_j(x_j^*(t))\}\,;\ee
\item[(ii)]
if $v=\tau_i$ is a diverging junction, then 
\be\label{diverge}\ba{c}z_i^*(t)=\gamma_i^F(t)d_i(x^*_i(t))\\[10pt] 
\gamma_i^F(t)=\sup\left\{\gamma\in[0,1]:\,\gamma R_{ik}d_i(x^*_i(t))\le s_k(x^*_k(t))\,,\ \forall k\in\mc E\right\}\,;\ea\ee
\item[(iii)] if $v=\tau_i$ is a merging junction with downstream cell $j$ (with $\sigma_j=\tau_i$), then 
\be\label{merge}(z^*_h(t))_{h:\,\tau_h=v}\in\argmax_{\substack{\\[2pt]\ds 0\le z_h\le d_h(x^*_h(t))\\[5pt]\ds \sum_{\substack{h \in\mc E:\\\tau_h=v}} z_h\le s_j(x^*_j(t))}}\sum_{\substack{h\in\mc E:\\\tau_h=v}}\kappa_h(t)z_h(t)\,.\ee
\end{enumerate}
\end{proposition}

Proposition \ref{prop:optimal-control-FNC} implies that at diverging junctions of a transportation network with total traffic volume as cost, affine supply functions and linear demand with identical slope, the optimal solution satisfies  \eqref{diverge}
which coincides with the FIFO diverge rule of Daganzo's cell transmission model. Similarly, for merging junctions with two upstream cells $h$ and $i$ and downstream cell $j$, \eqref{merge} is equivalent to  Daganzo's priority rule
	$$
		f^*_{ij}=\mid\{d_i(x_i^*),s_j(x_j^*)-d_h(x_h^*),p_i s_j(x_j^*)\}
	$$ 
and
	$$ 
		f^*_{hj}=\mid\{d_h(x_h^*),s_j(x_j^*)-d_i(x_i^*),p_h s_j(x_j^*)\}\,,
	$$
where $\mid\{a,b,c\}$ denotes the median and the priority parameters $p_h$ and $p_i$ are such that:
$p_i=1$ and $p_h=0$ if $\kappa_i>\kappa_h$; $p_i=0$ and $p_h=1$ if $\kappa_i<\kappa_h$; and $p_i=1-p_h$ is arbitrary in $[0,1]$ if $\kappa_i=\kappa_h$.\footnote{In fact, when $\kappa_i=\kappa_h$, the proportional rule $f_{ij}=d_i(x_i^*)\min\{1,s_j(x_j^*)/(d_i(x_i^*)+d_h(x_h^*))\}$, $f_{hj}=d_h(x_h^*)\min\{1,s_j(x_j^*)/(d_i(x_i^*)+d_h(x_h^*))\}$ also satisfies \eqref{merge}.}

Proposition~\ref{prop:optimal-control-FNC} and the discussion following it imply that, if the demand and supply functions are linear with identical slope, the optimal solution of the  FNC problem \eqref{FNC-0} with total traffic volume as a cost, 
is achieved by the CTM with FIFO rule at the diverge junctions with no additional control required, i.e., $\alpha_i(t)=1$, at all the cells $i$ immediately upstream of a diverge or an ordinary junction, while additional control (i.e., $\alpha_i(t)\le1$) might be useful on the cells immediately upstream the merge junctions. Observe that, for costs different than the total traffic volume, additional control might be beneficial at any cell, being it immediately upstream a merge, diverge, or ordinary junction.  

\section{Robustness analysis}
\label{section:robustness}

In Section~\ref{section:DTA_problems}, we discussed how the DTA \eqref{DTA-0} and the FNC \eqref{FNC-0} can be cast as convex optimization problems  \eqref{DTA-1} and \eqref{FNC-1} whose solutions can be mapped back into optimal demand control parameters $\alpha(t)$ and routing matrix $R(t)$ as explained in Proposition \ref{proposition:implementation}. In order to solve  \eqref{DTA-1} and \eqref{FNC-1} and then find compute the optimal $\alpha(t)$ and $R(t)$, one needs precise information about the input parameters, i.e., the initial traffic volume $x^0$ and exogenous inflow vector $\lambda(t)$ over the planning horizon $[0,T]$. However, in practice, information about these quantities inevitably involves uncertainties. Therefore, a reasonable strategy is to (i) compute the optimal solution of \eqref{DTA-0} or \eqref{FNC-0} for \emph{nominal} values of the exogenous inflows and initial traffic volume, and (ii) compute the corresponding nominal control inputs using Proposition~\ref{proposition:implementation}; and then expect the trajectory  under actual parameter values and the nominal control inputs to be close enough to the nominal trajectory. In this section, we provide formal guarantees on the robustness of a general system trajectory to perturbations in the initial traffic volume and the exogenous inflows, while maintaining the same control input.\footnote{We recall that the control inputs are open loop, and not in feedback form.} 
When specialized to the system trajectory corresponding to the optimal solution of the DTA \eqref{DTA-0} or the FNC \eqref{FNC-0} under nominal values of the initial traffic volume and exogenous inflows, this gives the desired result on robustness analysis with respect to uncertainties in the input parameters.  

We shall use the notational convention that $x^0$ and $\lambda(t)$ denote the \emph{nominal} values of the initial traffic volume and the exogenous inflows, while $\tilde{x}^0$ and $\tilde{\lambda}(t)$ 
denote the perturbed values of these parameters. Similarly $x(t)$ and $\tilde{x}(t)$ will, respectively, denote the nominal and perturbed trajectories, under the same open-loop control inputs $\alpha(t)$ and $R(t)$. Our robustness analysis will provide bounds on perturbations in the state trajectory due to perturbations in the inflow $\tilde{\lambda}-\lambda$ and the perturbations in initial traffic volume $\tilde{x}^0-x^0$. The bounds derived 
in the general setting, when specialized to the case where $x(t)$ corresponds to the optimal solution of the DTA \eqref{DTA-0} or the FNC \eqref{FNC-0},  and where $\alpha(t)$, and $R(t)$ are derived from $x(t)$ according to Proposition~\ref{proposition:implementation}, will give the desired robustness analysis for optimal solutions. 

Our technique relies on leveraging a certain \emph{monotonicity} property of the dynamical system underlying the dynamics in \eqref{dynamics},\eqref{flow-const},\eqref{R-const}-\eqref{initial-const}. Monotone systems are dynamical systems whose trajectories preserve the partial order between initial traffic volumes and external inputs. Specifically, for monotone systems, trajectories with initial traffic volumes\footnote{Here by $a\leq b$ for vectors in $\mathbb{R}^\E$ we mean $a_i\leq b_i$ for all $i\in\E$.} $x^0\leq\tilde x^0$, and inflows $\lambda(t)\leq\tilde\lambda(t)$ for all $t\in[0,T]$, are such that $x(t)\le\tilde x(t)$ for all $t\in[0,T]$. For given $\alpha(t)$ and $R(t)$, let us rewrite the dynamics in \eqref{dynamics}-\eqref{flow-const}, \eqref{DNL1}-\eqref{eq:gammaFIFO} compactly as:
\begin{equation}
\label{eq:dynamics-succinct}
\dot{x}_i = \lambda_i(t) + g_i(x,\alpha,R)
\end{equation}

A standard result in dynamical systems theory, known as Kamke's theorem \cite[Theorem 1.2]{HirschPS03}, 
then implies that \eqref{eq:dynamics-succinct} is monotone in a certain domain $\mc D\subseteq\RR_+^{\mc E}$ if and only if 
	\begin{equation}
		\label{eq:monotonicity}
		\frac{\partial g_i}{\partial x_j}(x, \alpha, R)\geq 0,\qquad \forall i\neq j\in\mc E\,, 
	\end{equation}
for almost every $x\in\mc D$, at all $t \in [0,T]$. 
It has been recognized, e.g., see \cite{CooganACC14}, that, for a given $\alpha(t)$ and $R(t)$, the dynamical system \eqref{eq:dynamics-succinct} is monotone, if $x(t)$ in the free-flow region, i.e., it satisfies $x \in \RR_+^{\mc E}$ where $\sum_{i\in\mc E} R_{ij} \ov d_i(x_i, \alpha_i) \leq s_j(x_j)$ for all $j \in \mc E$ for all $t \in [0,T]$.
%
However, the monotone property is not satisfied in the region where there exists at least one cell $k$ outgoing from a diverge junction (i.e., a cell $j$ such that there are other cells $k\ne j$ with $\sigma_k=\sigma_j$) that is congested, i.e., 
	$
		{\sum_{i\in\mc E}R_{ij} \ov d_i(x_i,\alpha_i)}>{s_j(x_j)}\,.
	$
In fact, in such a configuration, monotonicity is lost since an increase of the traffic volume $x_j$ in a congested cell $j$ has the effect of reducing its supply $s_j(x_j)$, and thus the coefficient $\gamma^F_i$ (as defined in \eqref{eq:gammaFIFO}) of any upstream cell $i$ such that $(i,j)\in\mc A$. In turn, this implies a reduction of the flow $f_{ik}$ from any such cell $i$ to any other downstream cell $k$ such that $(i,k)\in\mc A$, hence a reduction in the inflow on cell $k$. 

Given an $n$-dimensional vector $x$, recall that its $\ell_1$ norm is given by $\|x\|_1 := \sum_{i =1}^n |x_{i}|$.
For example, $\|\tilde{x}^0-x^0\|_1 = \sum_{i \in \E} |\tilde{x}^0_i-x^0_i|$, and $\|\tilde{x}(t)-x(t)\|_1=\sum_{i \in \E} |\tilde{x}_i(t) - x_i(t)|$ for all $t \in [0,T]$. We let $\| \lambda - \tilde{\lambda} \|_{1,t} := \int_0^t \|\lambda(s)-\tilde{\lambda}(s)\|_1 \, ds$.

\begin{proposition}
\label{theo:sensitivity}
If the perturbations $\|\tilde{x}^0 - x^0\|_1$ and $\|\tilde{\lambda}-\lambda\|_{1,T}$ are sufficiently small,
then 
\begin{equation}
		\label{eq:nonFIFOResultTraj}
		\|\tilde{x}(t) - x(t)\|_1 \leq \|\tilde{x}^0 - x^0\|_1 + \|\tilde{\lambda}-\lambda\|_{1,t} \, , \qquad  \forall \, t \in [0,T].
	\end{equation}	
\end{proposition}

The sufficiently small condition in Proposition~\ref{theo:sensitivity} is to ensure that the perturbed trajectories remain in free-flow, and hence the system is monotone along them. 
The bound provided by Proposition~\ref{theo:sensitivity} has the desirable property that the right hand side of \eqref{eq:nonFIFOResultTraj} goes to zero as the magnitude of perturbations go to zero. However, the bound is conservative for non-negligible values of perturbations. The next result provides a tighter bound in this latter case. In preparation for it, let 
$\|\lambda_i - \tilde{\lambda}_i\|_{\infty}:=\sup_{t \in [0,T]} |\lambda_i(t)-\tilde{\lambda}_i(t)|$, $i \in \R$, and 
\begin{equation*}
	\label{eq:perturbed-lambda-def}
 	\bar{\lambda}_i := \sup_{t\in[0,T]}\{\lambda_i(t) + \|\lambda_i-\tilde{\lambda}_i\|_{\infty}\}, \!\!\quad  \!\!
 	\underline{\lambda}_i:=\max\{0,\inf_{t\in[0,T]}\{\lambda_i(t) - \|\lambda_i-\tilde{\lambda}_i\|_{\infty}\}\}, \!\!\quad i \!\!\in \R
	\end{equation*}
\begin{equation*}
	\label{eq:perturbed-x0-def}
	\bar{x}^0_i:= x^0_i + |x^0_i-\tilde{x}^0_i|, \quad  
	\underline{x}^0_i:= \max\{0,x^0_i - |x^0_i-\tilde{x}^0_i|\}, \quad i \in \E
\end{equation*}

Note that the extreme values of perturbed inflows, $\bar{\lambda}$ and $\underline{\lambda}$, are constant. The next bound is expressed in terms of equilibrium traffic volumes at these extreme values, denoted as $x^{\text{eq}}(\bar{\lambda})$ and $x^{\text{eq}}(\underline{\lambda})$. Existence and uniqueness of these equilibria are discussed in \cite{CooganACC14}. In particular, if an equilibrium exists, it is unique.  
 
\begin{proposition}
\label{thm:robustness}
If $\|\tilde{x}^0 - x^0\|_1$ and $\|\tilde{\lambda}-\lambda\|_{1,T}$ are sufficiently small, and $x^{\text{eq}}(\bar{\lambda})$ and $x^{\text{eq}}(\underline{\lambda})$ exist, then
	\begin{equation}
	\label{eq:robustness-bound}
	\begin{split}
		\|\tilde{x}(t) - x(t)\|_1 
		&	\leq 	\|x^{\text{eq}}(\bar{\lambda}) - x^{\text{eq}}(\underline{\lambda})\|_1 + \|\bar{x}^0 - \underline{x}^0\|_1 +\\
		&	\qquad\qquad \min_{\xi \in \bar{x}^0, \underline{x}^0}\{\|x^{\text{eq}}(\underline{\lambda}) - \xi\|_1 + \|x^{\text{eq}}(\bar{\lambda}) - \xi\|_1\}\, \qquad\forall t \in [0,T].
	\end{split}
	\end{equation}
\end{proposition}

Again, in Proposition~\ref{thm:robustness}, the smallness of the perturbations is required to ensure that the perturbed trajectories remain in the free-flow region and that the corresponding extreme values of inflows admit equilibria.

\begin{remark}
\label{rem:optimal-monotone}
The optimal solution $x(t)$ to the DTA \eqref{DTA-0} or the FNC \eqref{FNC-0} under the control policies $(\alpha, R)$ designed in Proposition~\ref{proposition:implementation} is in the free flow region. Hence, Propositions~\ref{theo:sensitivity} and \ref{thm:robustness} are applicable when the nominal trajectory $x(t)$ corresponds to the evolution of the system trajectory under the control policies designed in Proposition~\ref{proposition:implementation}. 
\end{remark}

Note that the right hand side of \eqref{eq:robustness-bound} does not go to zero as perturbations go to zero. However, for non-negligible perturbations, it gives less conservative upper bound than \eqref{eq:nonFIFOResultTraj}. Therefore, we propose to use the bound equal to the minimum of the right hand sides of \eqref{eq:nonFIFOResultTraj} and \eqref{eq:robustness-bound}. Indeed, this is the upper bound we use for comparison with simulation results.

The perturbations for which Propositions~\ref{theo:sensitivity} and \ref{thm:robustness} are not applicable typically result in an \emph{overload} condition, where the perturbed exogenous inflow exceeds network capacity. In such a scenario, the traffic volume on the sources grow unbounded. In our previous work~\cite[Proposition 2]{Como.Lovisari.ea:TCNS15}, we characterized this growth rate for a related dynamical network flow model. Inspired by our previous work, we suggest a simple upper bound in the overload regime, for networks with a single source, e.g., as shown in Figure~\ref{figure:NetworkZiliaskopoulos}, and with constant values for nominal and perturbed exogenous inflow. The suggested upper bound will be compared in the next Section against simulation results. Let $\hat{\lambda}$ be the supremum of all perturbed exogenous inflows under which the perturbed trajectories remain in free-flow, and hence for which Propositions~\ref{theo:sensitivity} and \ref{thm:robustness} are applicable. Let $\hat{x}$ denote the perturbed trajectory under perturbed inflow $\hat{\lambda}$. Triangle inequality implies that 
\begin{equation}
\label{eq:triangle-ineq}
\|\tilde{x}(t)-x(t)\|_1 \leq \|\hat{x}(t)-x(t)\|_1 + \|\tilde{x}(t)-\hat{x}(t)\|_1 \quad \forall \, t \geq 0. 
\end{equation}
The first term in \eqref{eq:triangle-ineq} can be upper bounded using Propositions~\ref{theo:sensitivity} and \ref{thm:robustness}. In the second term, the total traffic volume on sources grows unbounded at a rate which is expected to be equal to $||\tilde{\lambda}-\hat{\lambda}||_1 = \sum_{i\in\mc S}|\tilde{\lambda}_i-\hat{\lambda}_i|$, i.e., 
\begin{equation}
\label{eq:growth-rate}
\limsup_{t \to \infty} \frac{\|\tilde{x}(t)-\hat{x}(t)\|_1}{t} = ||\tilde{\lambda} - \hat{\lambda}||_1 \, .
\end{equation}
Since \eqref{eq:growth-rate} gives a reasonable upper bound for large $t$, considering the relatively large time horizon for the simulations, we shall use 
\begin{equation}
\label{eq:growth-approx}
\|\tilde{x}(t)-\hat{x}(t)\|_1 \leq ||\tilde{\lambda}-\hat{\lambda}||_1\, t
\end{equation}
for all $t$. In summary, the upper bound for the overload regime is obtained from \eqref{eq:triangle-ineq} and \eqref{eq:growth-approx}.

The upper bounds in \eqref{eq:nonFIFOResultTraj}, \eqref{eq:robustness-bound} and \eqref{eq:growth-approx} are to be contrasted with the bounds obtained from standard sensitivity analysis of ordinary differential equations:
\begin{equation}
\label{eq:sensitivity-analysis-bound}
\|\tilde{x}(t)-x(t)\|_1 \leq e^{L_g t} \left(\int_0^t e^{-L_g s} \| \tilde{\lambda}(s) - \lambda(s) \|_1 \, ds + \| \tilde{x}(0) - x(0)\|_1 \right), \, \, t \in [0,T]
\end{equation}
for any $L_g > 0$ such that, for all $t \in [0,T]$, 
\begin{equation}
\label{eq:global-lipschitz}
\|g(\tilde{x},\alpha(t),R(t)) - g(x,\alpha(t),R(t))\|_1 \leq L_g \|\tilde{x}-x\|_1 \, , \quad \forall \, \tilde{x}, x \in \Pi_{i \in \mc E}[0,x^{\textrm{jam}}_i]
\end{equation}
The derivation of \eqref{eq:sensitivity-analysis-bound} is provided in Section~\ref{sec:sensitivity-analysis}. A particular choice of $L_g$ is also provided in Lemma~\ref{lem:lipschitz-constant} in Section~\ref{sec:sensitivity-analysis}. If $\|\tilde{\lambda}(t)-\lambda(t)\|_1$ is constant over $[0,T]$, then \eqref{eq:sensitivity-analysis-bound} becomes
\begin{equation}
\label{eq:sensitivity-bound-x-constant-delta-lambda}
\|\tilde{x}(t)-x(t)\|_1 \leq \frac{e^{L_gt} - 1}{L_g} \|\tilde{\lambda} - \lambda\|_1 + e^{L_g t} \|\tilde{x}(0) - x(0)\|_1, \, \, t \in [0,T]
\end{equation}

The bounds that we have obtained are very easy to compute and yet provide tight upper limits on the cost displacement as a function of the inflow uncertainty, as we illustrate via numerical examples in Section~\ref{subsec:robustnessBounds}.

\subsection{Extensions to Non-FIFO Models}
The proportional rule in \eqref{eq:gammaFIFO} for traffic flow under congestion at nodes with multiple outgoing links is often also referred to as the FIFO rule. While such FIFO rules 
are a natural framework for multi-origin multi-destination transportation networks, it has been pointed out by several authors, e.g., see \cite{DoanTRB:12}, that it corresponds to a rather conservative behavioral model, in which not only drivers never change their routing choice, thus blindly queuing up even in presence of alternative routes to the same destination, but more importantly it does not take into account presence of multiple lanes for multiple maneuvres at junctions. For example, a congested offramp on a freeway would slow down and possibly block the flow of vehicles on the main line, which is not always realistic. 
Recently, e.g., see \cite{Karafyllis.Papageorgiou:TCNS14}, there has also been interest in  non-FIFO  rules, for which \eqref{eq:gammaFIFO} is replaced by
\begin{equation}
		\label{eq:gammanonFIFO}
		\gamma_j^{N} = 
		\sup\left\{\gamma\in[0,1]:\, \gamma\cdot{\sum_{h\in\mc E}\ov R_{hj} \ov d_h(x_h)}\leq s_j(x_j)\right\}
	\end{equation}
The main difference of the non-FIFO model with respect to the FIFO one is that congestion in one of the outgoing cells does not influence the flow towards other outgoing cells. In the previous example, while the congested offramp forces vehicles that would like to take it to stop in the freeway, those that want to continue on the main line are free to do so, if the downstream supply of the main line is sufficient. 
The non-FIFO model is best suited for single-origin single-destination network (in which the actual path of a vehicle does not matter), but could possibly exhibit unrealistic behavior in multi-origin multi-destination networks. A possibly more realistic model may involve a combination of FIFO and non-FIFO, in which a fraction of drivers can change path, while another fraction cannot -- for example, private cars can deviate from their path to chose a more convenient one, while buses or trucks have prescribed paths to follow.

Traffic systems under non-FIFO models can be shown to be monotone almost everywhere in $\mc D=\RR_+^{\mc E}$, e.g., see \cite{Lovisari.Como.ea:CDC14}. This is because an increase in traffic volume in an outgoing link does not decrease the flow towards other outgoing links at the same junction, as thoroughly discussed in \cite{Como.Lovisari.ea:TRb16-extended}. As such, their stability properties and their set of equilibria are substantially different from a traffic system under FIFO model \cite{Lovisari.Como.ea:CDC14}. However, several results that we stated for FIFO models can be naturally extended to non-FIFO models. In fact, the control design derived in Proposition~\ref{proposition:implementation} holds true for traffic dynamics \eqref{dynamics}-\eqref{flow-const}, \eqref{DNL1}, \eqref{eq:fjFIFO} and \eqref{eq:gammanonFIFO} with no difference, since the resulting optimal trajectory is in free-flow, where $\gamma^F = \gamma^N$ and where thus FIFO and non-FIFO models coincide. Moreover, the robustness analysis of Section~\ref{section:robustness} can be extended, for a given $\alpha(t)$ and $R(t)$, as follows. Since the system is monotone everywhere in $\mc D=\RR_+^{\mc E}$, the robustness bound in \eqref{eq:nonFIFOResultTraj} holds true independent of the magnitude of perturbations $\|\tilde{x}^0 - x^0\|_1$ and $\|\tilde{\lambda}-\lambda\|_{1,T}$. In our previous work~\cite{Lovisari.Como.ea:CDC14}, we have shown that, under the non-FIFO model, in general, there exists a manifold of equilibria for constant exogenous inflows $\lambda$. Nevertheless, the upper bound in \eqref{eq:robustness-bound} is valid for any choice of equilibrium. Numerical simulations reported in Figure \ref{figure:cost_function_of_lambda_constantinflow_FIFO} and \ref{figure:cost_function_of_lambda_constantinflow_nonFIFO} suggest that the non-FIFO dynamics tend to be more robust to perturbations than the FIFO one. 

\section{Numerical studies}
\subsection{Simulations Comparing Various DTA Formulations}
\label{sec:simulations-implementation}

In this section, we compare the optimal solutions obtained from the two proposed relaxations of the DTA problem \eqref{DTA-1} and of the FNC problem \eqref{FNC-1}, with an uncontrolled traffic simulation using FIFO model, following \cite{Daganzo:95}, but with proportional merge rule as in \cite{MuralidharanACC12, CooganACC14}, as described in Section \ref{section:traffic_model}. We remark once again that the optimal solutions are not meant to represent, on their own, traffic simulations, but rather optimal trajectories, which can be made feasible with respect to the traffic dynamics  via the control signals that are designed in Proposition~\ref{proposition:implementation}.

We first provide the details of the simulation setting that are common for all our numerical studies.
We solve the convex optimization problems in MatLab using the Convex Programming package \texttt{cvx} \cite{CVXSoftware, GrantRALC08}. We use the single-source single-sink network described in \cite{Ziliaskopoulos:00} and shown in Figure~\ref{figure:NetworkZiliaskopoulos} for our simulations.
\begin{figure}
\centering 
\includegraphics[width=9cm]{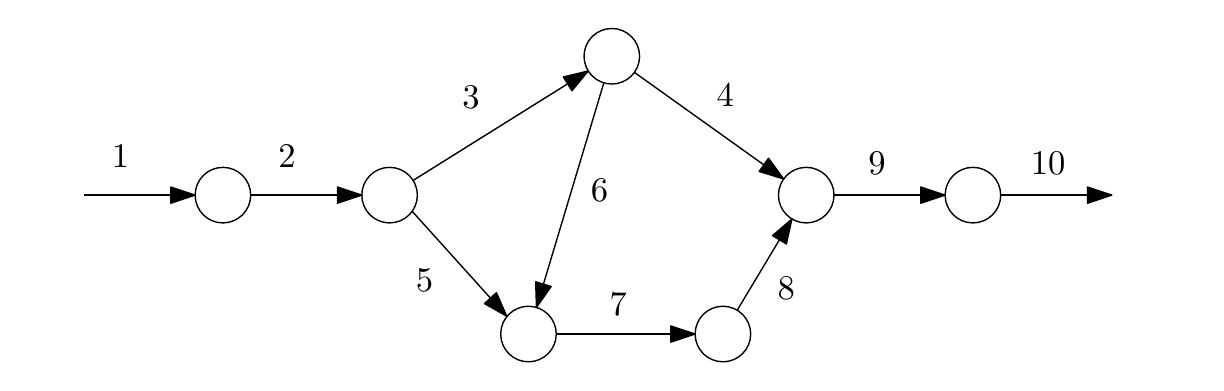}
\caption{The network used in the numerical study.}
\label{figure:NetworkZiliaskopoulos}
\end{figure}
For implementation, we discretize the continuous formulation according to standard practices in Cell Transmission Models. Time is slotted with sampling time
$\tau = 10$ seconds. In all the cells, demand and supply functions are piecewise affine, with
	\begin{align*}
		d_i(x_i,t) = \frac{v_ix_i}{L_i}, \qquad
		s_i(x_i,t) = \min\left\{\frac{w_i({x^{jam}_i} - x_i)}{L_i}, C_i(t)\right\}		
	\end{align*}
where $v_i$, $w_i$, $C_i(t)$ ($t = 1, \dots, n$), $L_i$ and {$x^{jam}_i$} are the free-flow speed, the wave speed, the capacity, the length and the jam traffic volume on cell $i$, respectively. Recall that the actual demand is the saturated $\min\{d_i(x_i), C_i(t)\}$. The values of these parameters, along with number of lanes and length of cells, are specified in Table~\ref{table:parametersNetwork}. The units of all parameters, as well as of inflows, provided below, are chosen in such a way that physical consistency is ensured. In addition, with the chosen parameters, a vehicle travels along an entire cell in exactly one time slot at maximum speed, which in the considered scenario is the free-flow speed $v$. 
Therefore, the Courant-Friedrichs-L\'{e}évy condition $\frac{\tau \max_iv_i}{\min_iL_i} \leq 1$, which is necessary for numerical stability, is satisfied. 
	
Vehicles enter the network from cell $1$ at rate $\lambda_1(t)$. The setup in this Section and Section~\ref{sec:simulations-congestion} differs from the setup in Section~\ref{subsec:robustnessBounds} in terms of the time-varying values of $\lambda_1(t)$ and capacities on the links. In this Section and in Section~\ref{sec:simulations-congestion}, we consider a setting in which\footnote{Inflows are normalized by the onramp length as the dynamics involve the traffic volume of vehicles.} $\lambda_1(1) = 8$, $\lambda_1(2) = 16$, $\lambda_1(3) = 8$ and $\lambda_1(t) = 0$ for $t\geq 4$, with a time horizon of $T = 25$ steps, and in which the capacity in the cells is constant except on cell $4$, where a bottleneck is simulated by setting $C_4(t) = 6$ veh/$\tau$ for $t \neq 5,6,7,8$, $C_4(5) = C_4(6) = 0$ veh/$\tau$, $C_4(7) = C_4(8) = 3$ veh/$\tau$. Exogenous turning ratios are as follows: $R_{23} = 2/3$, $R_{25} = 1/13$ and $R_{34} = 2/3$, $R_{36} = 1/3$, the others being trivial. 

For these values of inflow and link capacities, we compute the optimal trajectories and costs for the DTA and FNC problems and the costs associated with a system evolving under FIFO traffic dynamics with proportional merge rule. For brevity, we refer to the last model as FIFO. The initial traffic volume for each case was $x(0)=0$. The results for the total traffic volume cost $\Psi^{(1)}(x) = \sum_t\sum_{i\in\mc E}x_i(t)$ are reported in Table~\ref{table:costsLinear}, whereas the results for the quadratic cost, $\Psi^{(2)}(x) = \sum_t\sum_{i\in\mc E}x_i^2(t)$, are reported in Table~\ref{table:costsQuadratic}. The corresponding trajectories for volume of vehicles, $x_i(t)$, for a few representative cells are shown in Figures~\ref{figure:trajectoriesLinear} and Figures~\ref{figure:trajectoriesQuadratic} for linear and quadratic cost, respectively. 

\begin{table}
\begin{center}
\begin{tabular}{l|l}
Parameter						&	Value	\\
\hline
Free-flow speed $v_i$, wave speed $w_i$		&	$50$ feet / sec							\\
Length of cell $L_i$			&	$500$ feet 								\\
Capacity $C_i$					&	$6\ell_i$ veh/$\tau$ (except $4$)			\\
Number of lanes $\ell_i$		&	$2$ for $i = 1,2,9,10$; $1$ otherwise	\\
Jam volume $x^{jam}$ 			&	$10\ell_i$ veh
\vspace{.1cm}
\end{tabular}
\end{center}
\caption{Cell parameters.}
\label{table:parametersNetwork}
\end{table}

As expected, the DTA, being the least constrained, gives a cost that is smaller than the FNC scheme, and both these optimal solutions yield a total cost that is no larger than the cost computed from the traffic simulation under FIFO traffic dynamics. This favorable comparison between the optimal costs for the DTA variants and the FIFO model served as a motivation to investigate the feasibility of optimal DTA solutions with respect to FIFO traffic dynamics, which we addressed in Proposition~\ref{proposition:implementation}.

Interestingly, for the linear cost criterion, the FNC optimal cost coincides exactly with the FIFO case, confirming Proposition~\ref{prop:optimal-control-FNC}, which shows that this no mere coincidence, and that one can identify a class of settings for which this property can be proven to be true. Finally, we refer to the cost comparison for the quadratic cost criterion in Table~\ref{table:costsQuadratic} to emphasize that this property does not hold true in general.

\def \WIDTHPAIRED {4}
\begin{figure}
\centering
\begin{tabular}{cc}
\includegraphics[width=\WIDTHPAIRED cm]{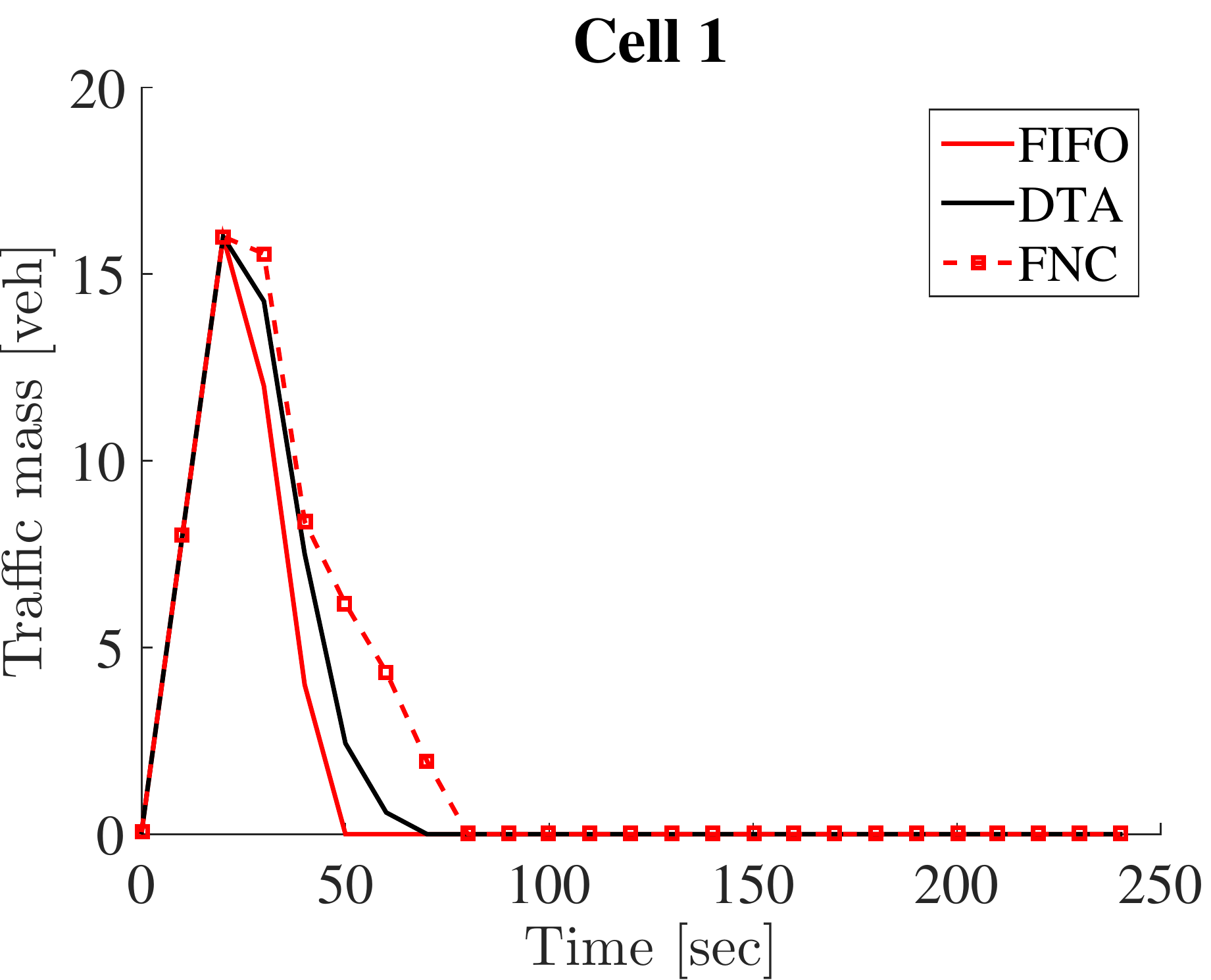}&\includegraphics[width=\WIDTHPAIRED cm]{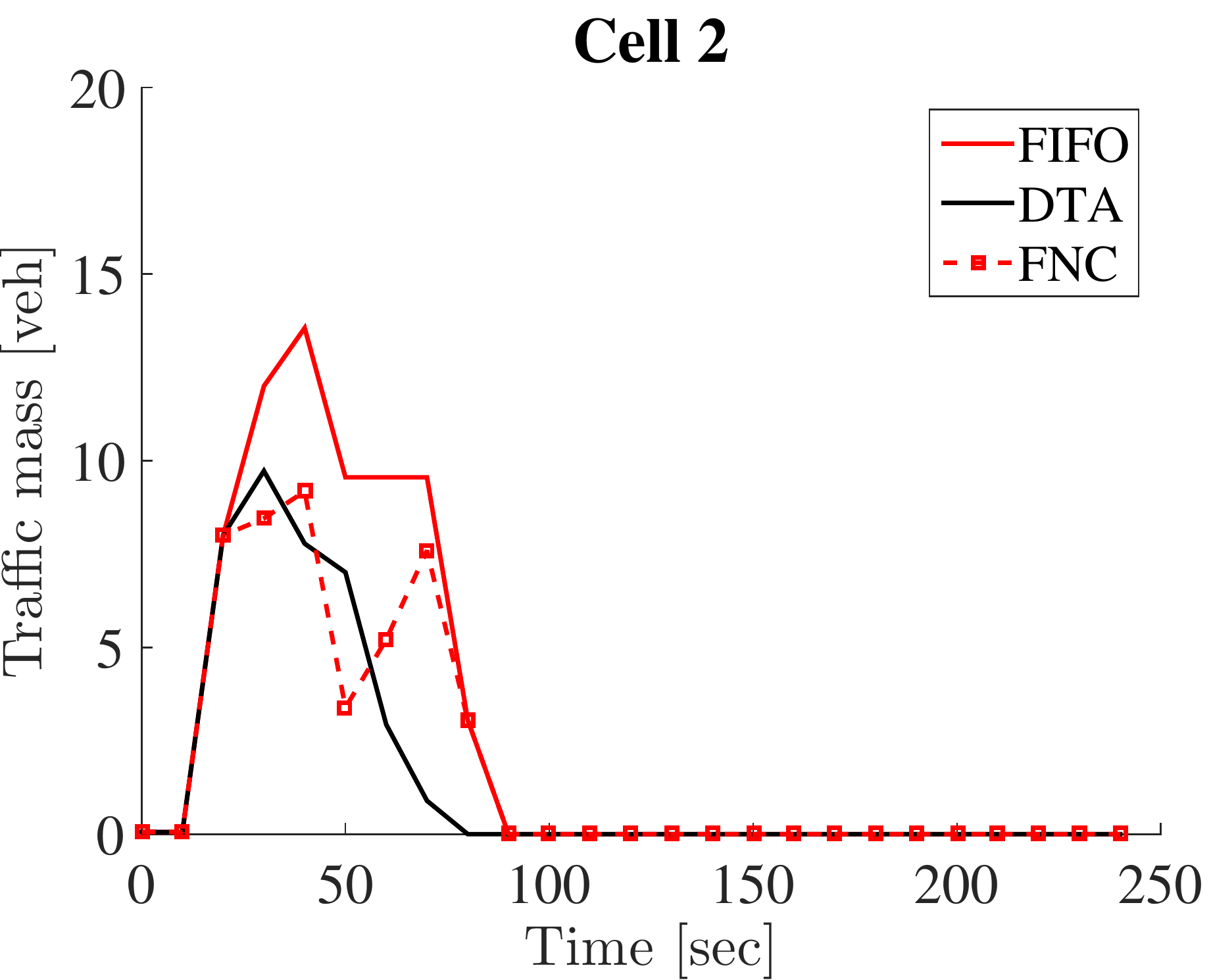}\\
\includegraphics[width=\WIDTHPAIRED cm]{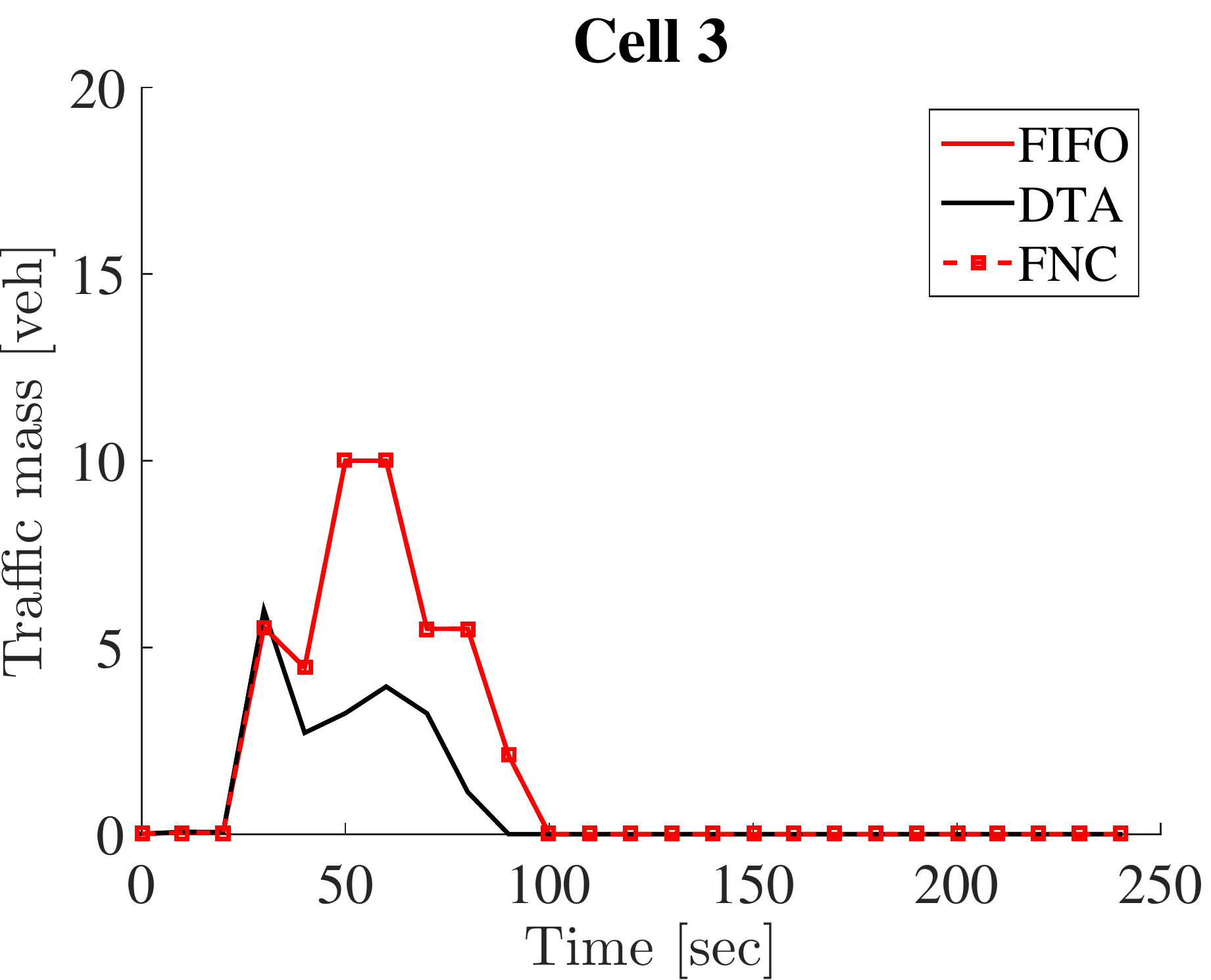}&\includegraphics[width=\WIDTHPAIRED cm]{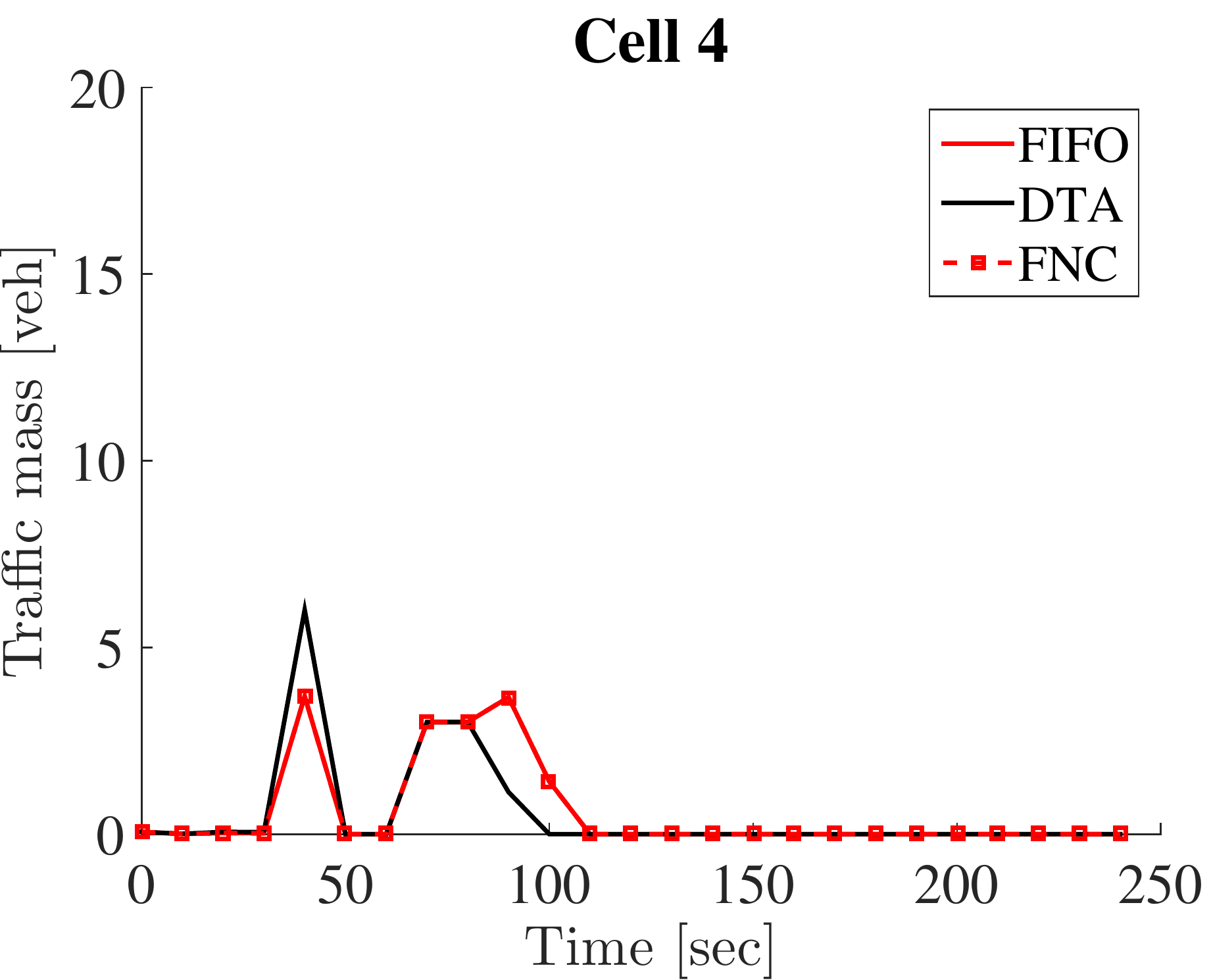}\\
\end{tabular}
\caption{Trajectories of the number of vehicles on cells $1,2,3,4$ for the system under FIFO traffic dynamics and for the optimal solutions corresponding to the two variants of DTA, for linear cost.}
\label{figure:trajectoriesLinear}
\end{figure}

\begin{table}
\begin{center}
\begin{tabular}{l|l}
Scheme &	Cost \\
\hline
FIFO		&	281.6	\\
DTA   		&	246		\\
FNC 			&	281.6	\\
\vspace{.1cm}
\end{tabular}
\end{center}
\caption{Comparison between optimal cost for the two DTA variants, and the cost for the system under FIFO traffic dynamics, for the linear cost criterion.}
\label{table:costsLinear}
\end{table}

\begin{figure}
\centering
\begin{tabular}{cc}
\includegraphics[width=\WIDTHPAIRED cm]{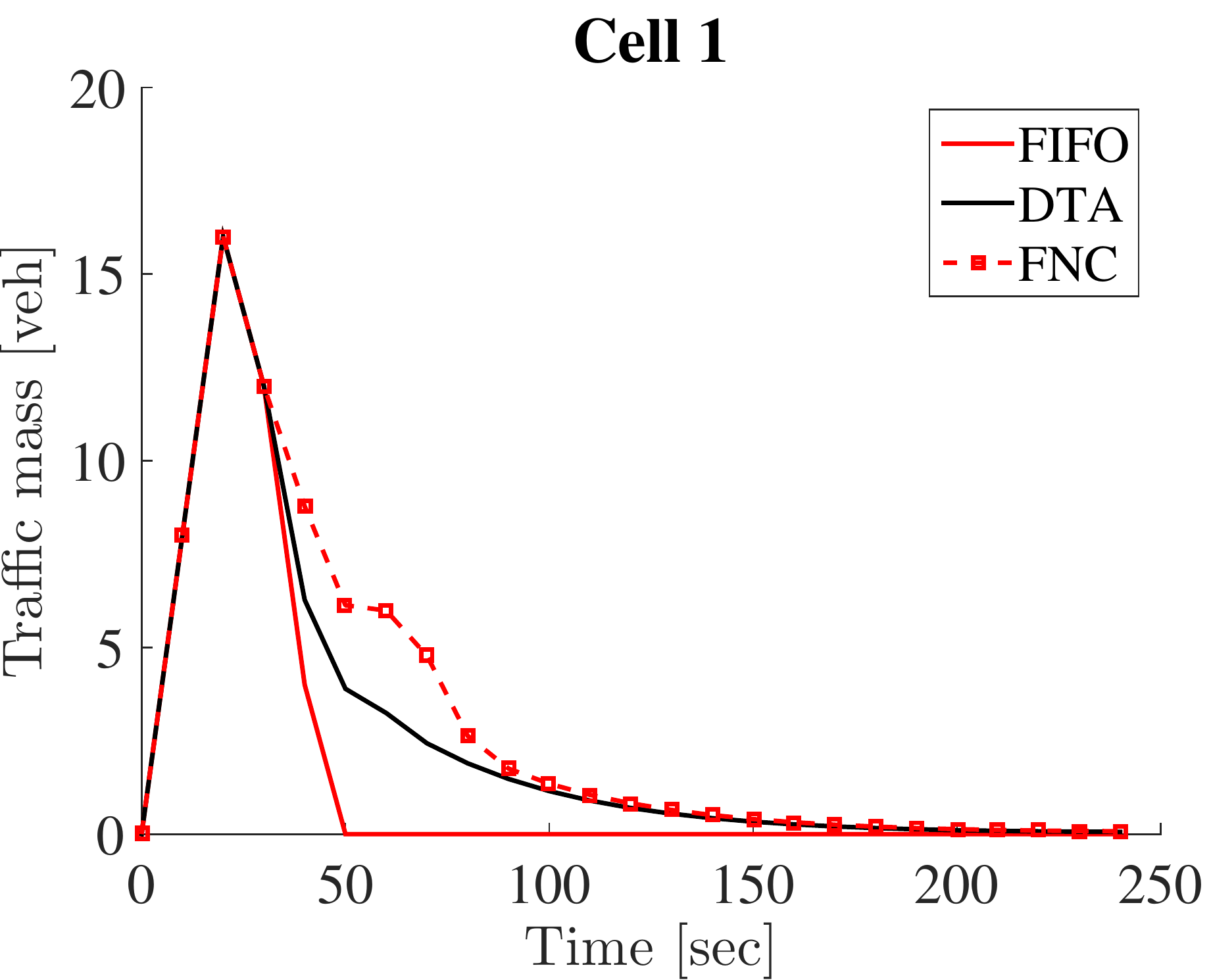}&\includegraphics[width=\WIDTHPAIRED cm]{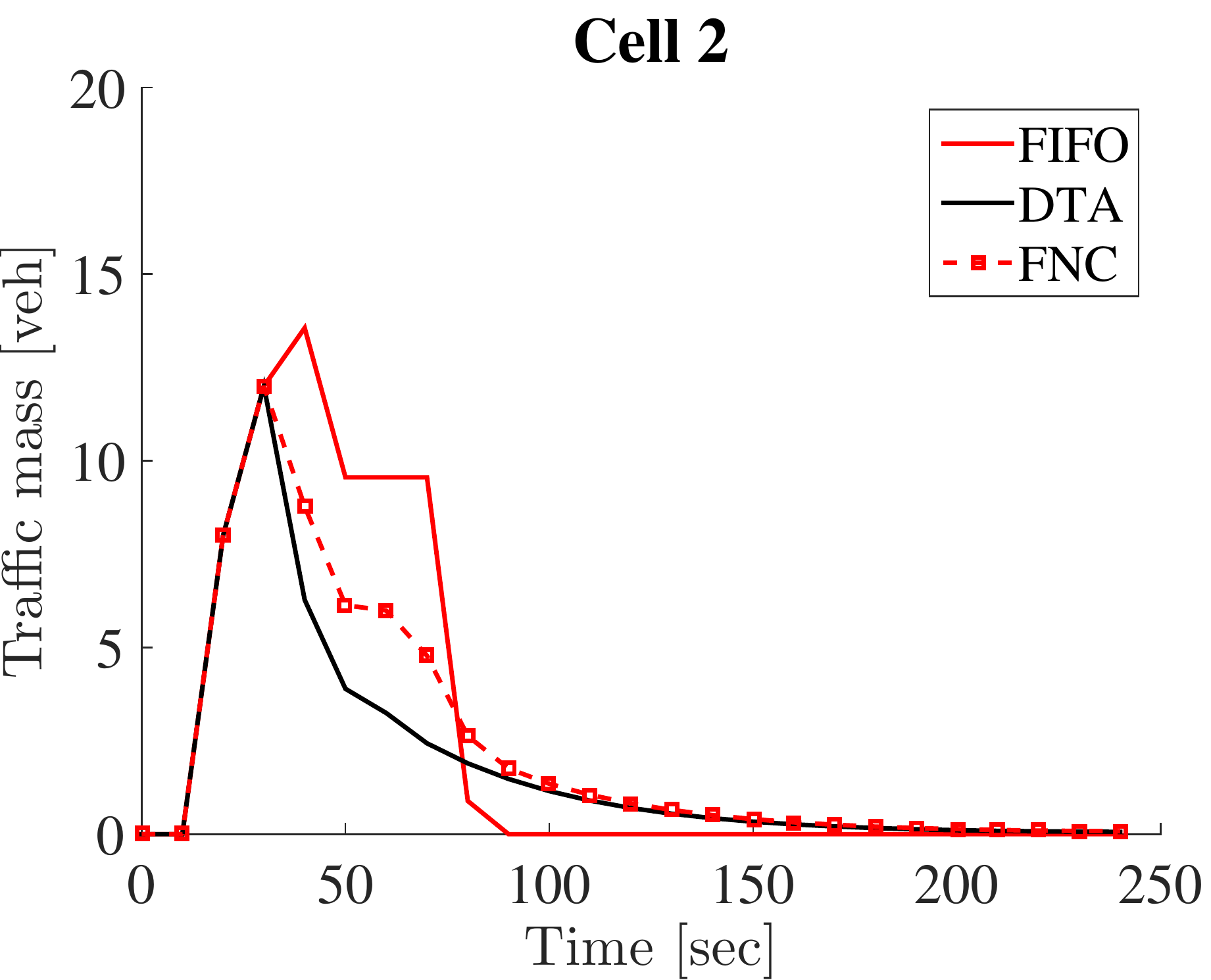}\\
\includegraphics[width=\WIDTHPAIRED cm]{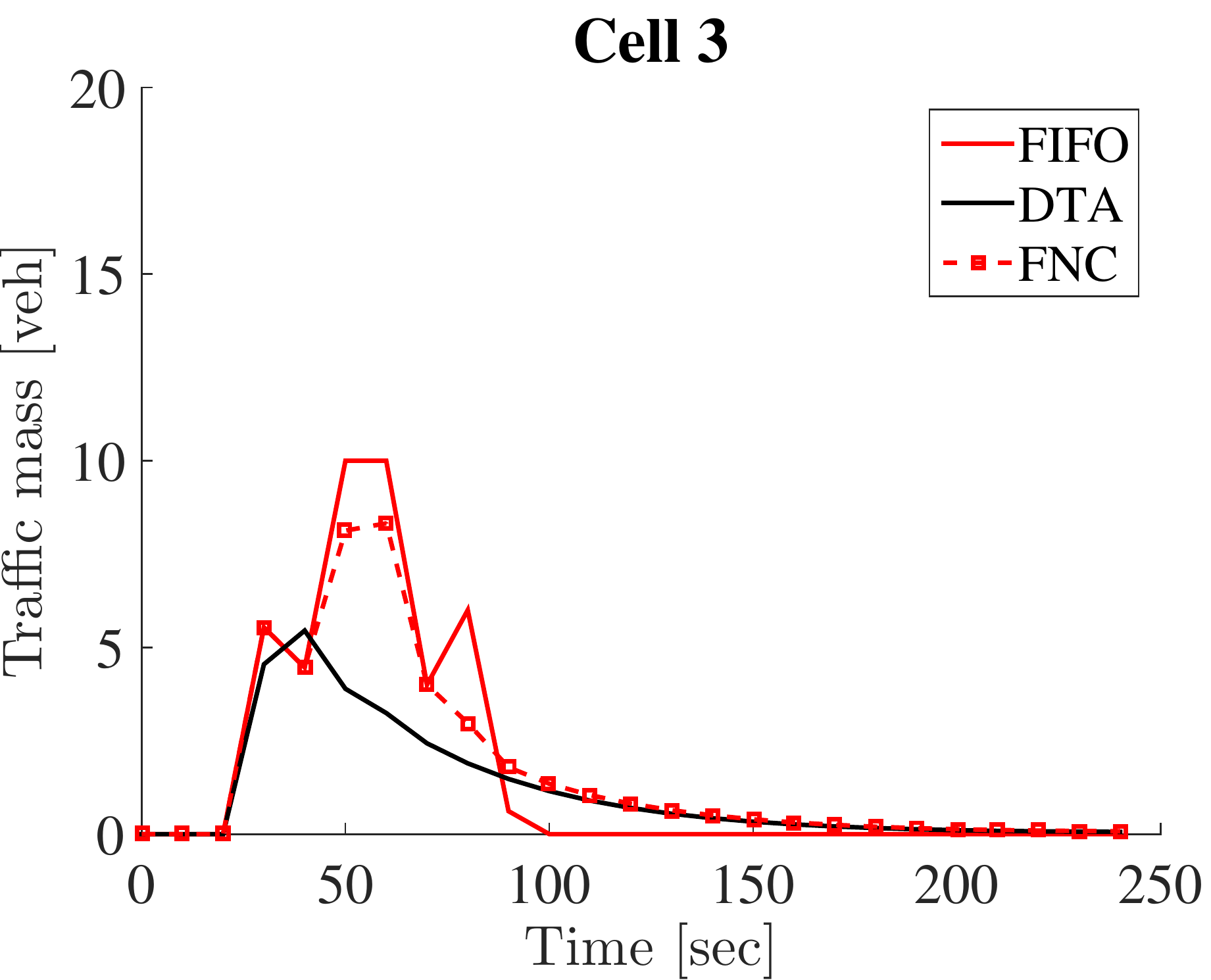}&\includegraphics[width=\WIDTHPAIRED cm]{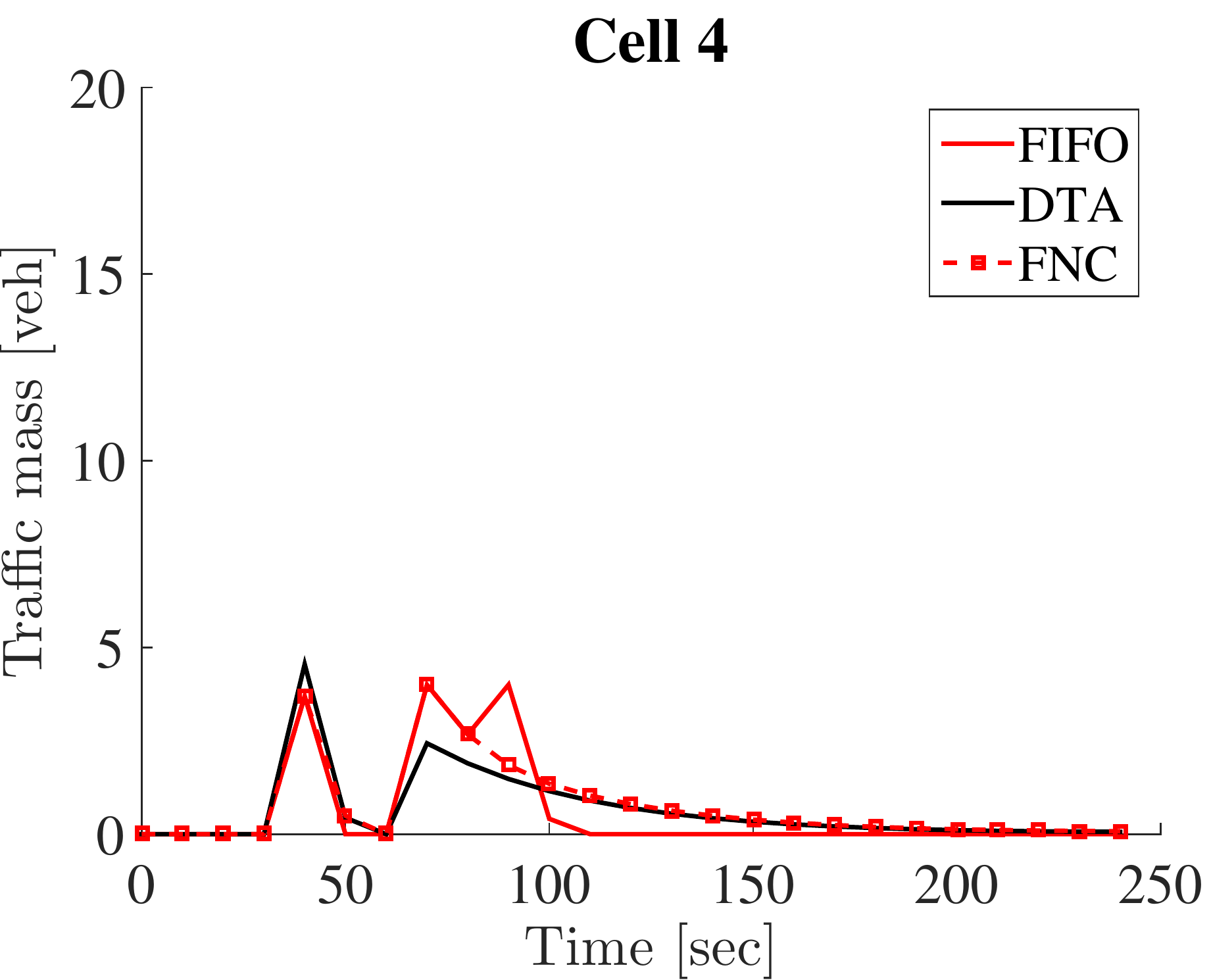}\\
\end{tabular}
\caption{Trajectories of the number of vehicles on cells $1,2,3,4$ for the system under FIFO traffic dynamics and for the optimal solutions corresponding to the two DTA variants, for quadratic cost.}
\label{figure:trajectoriesQuadratic}
\end{figure}

\begin{table}
\begin{center}
\begin{tabular}{l|l}
Scheme &	Cost \\
\hline
FIFO		&	1930.5	\\
DTA  		&	1393.5		\\
FNC 			&	1595.7	\\
\vspace{.1cm}
\end{tabular}
\end{center}
\caption{Comparison between optimal cost for the two DTA variants, and the cost for the system under FIFO traffic dynamics, for the quadratic cost criterion.}
\label{table:costsQuadratic}
\end{table}

\subsection{Robustness bounds}
\label{subsec:robustnessBounds}

In this second numerical study we compare the robustness bounds obtained in Propositions~\ref{theo:sensitivity} and \ref{thm:robustness}, and in \eqref{eq:triangle-ineq} and \eqref{eq:growth-approx}, with simulation results. We run simulations for a constant inflow scenario where the nominal value of the inflow is $\lambda_1(t) = 5$ for all $t = 0, 1, \dots, T$, $T = 200$, and in which the link capacities are as in Section~\ref{sec:simulations-implementation}, except that $C_4(t) = 6$ veh/$\tau$ for all $t$, i.e., we do not consider any bottleneck in this section.

For these parameters, and for the total traffic volume cost $\psi(x) = \sum_{i\in\mc E}x_i$, we solve the FNC with initial traffic volume $x(0) = 0$, and compute the corresponding controls, as given by Proposition~\ref{proposition:implementation}. We then compute the cost for the FIFO and non-FIFO dynamic traffic models under these controls, but under perturbed inflow $\tilde{\lambda}_1(t) = \lambda_1(t) + \Delta\lambda$ and, for simplicity, zero uncertainty in the initial traffic volume (namely, $x(0) = \tilde{x}(0)$). We emphasize that the controls do not change with $\Delta\lambda$. Let the optimal trajectory, i.e., the solution of FNC, be denoted as $x^*(t)$, and the trajectory under perturbed inflow $\Delta\lambda$ be denoted as $\tilde{x}^{(\Delta \lambda)}(t)$. (Note that $x^{0}(t)=x^*(t)$). We compute the resulting perturbation in cost $\Delta\Psi(\Delta\lambda) = \sum_{t=0}^{200}\sum_{i\in\mc E}(\tilde{x}_i^{(\Delta\lambda)}(t)-x^*_i(t))$ for various values of $\Delta\lambda$ in $[0, 3]$. 

We compared our bounds both for a system under FIFO traffic dynamics and for a system under non-FIFO traffic dynamics, and the results are shown in Figure~\ref{figure:cost_function_of_lambda_constantinflow_FIFO} and in Figure~\ref{figure:cost_function_of_lambda_constantinflow_nonFIFO}, respectively.

\begin{figure}
\centering
\begin{tabular}{cc}
\includegraphics[height=4.5 cm]{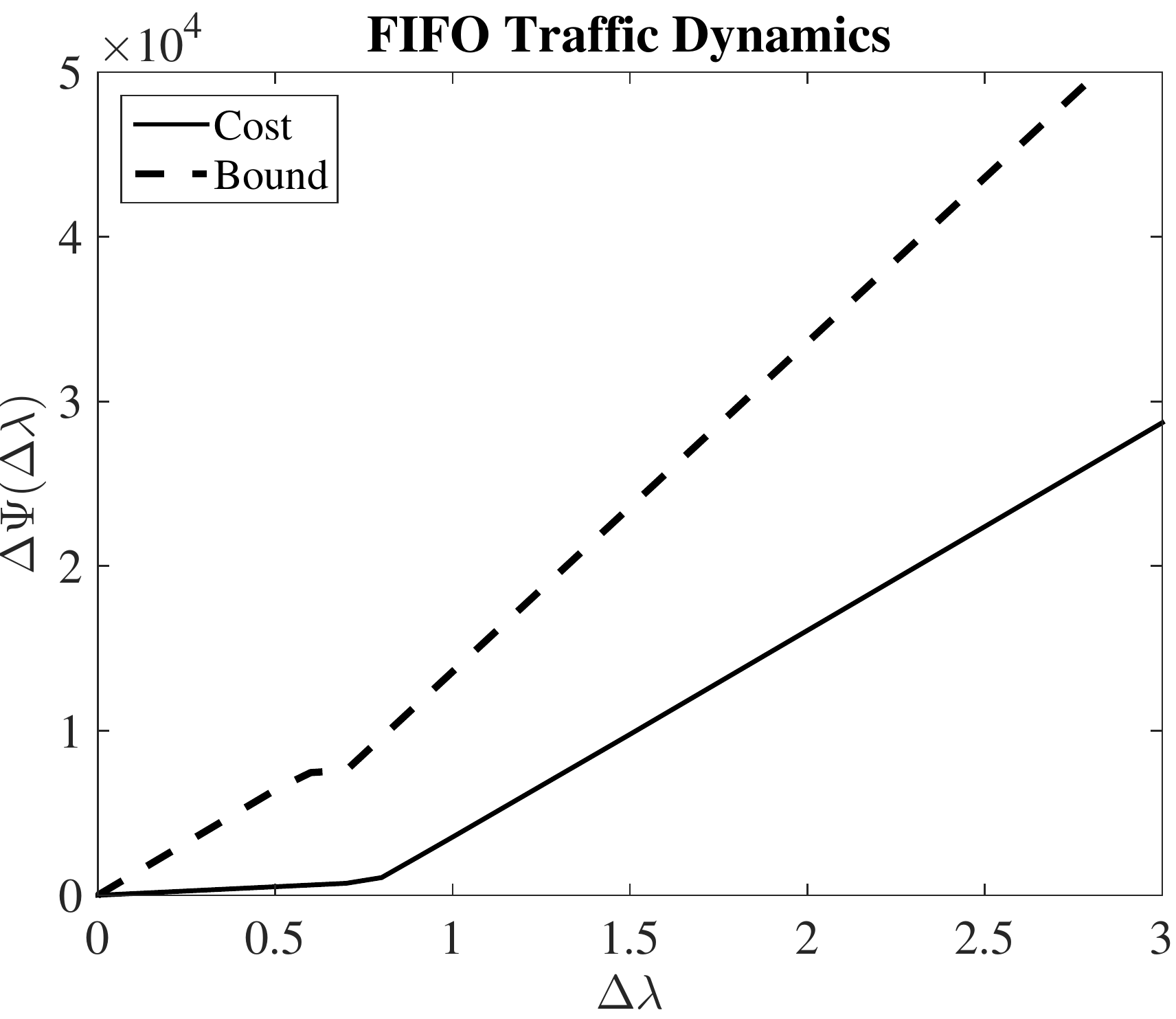}
&
\includegraphics[height=4.5 cm]{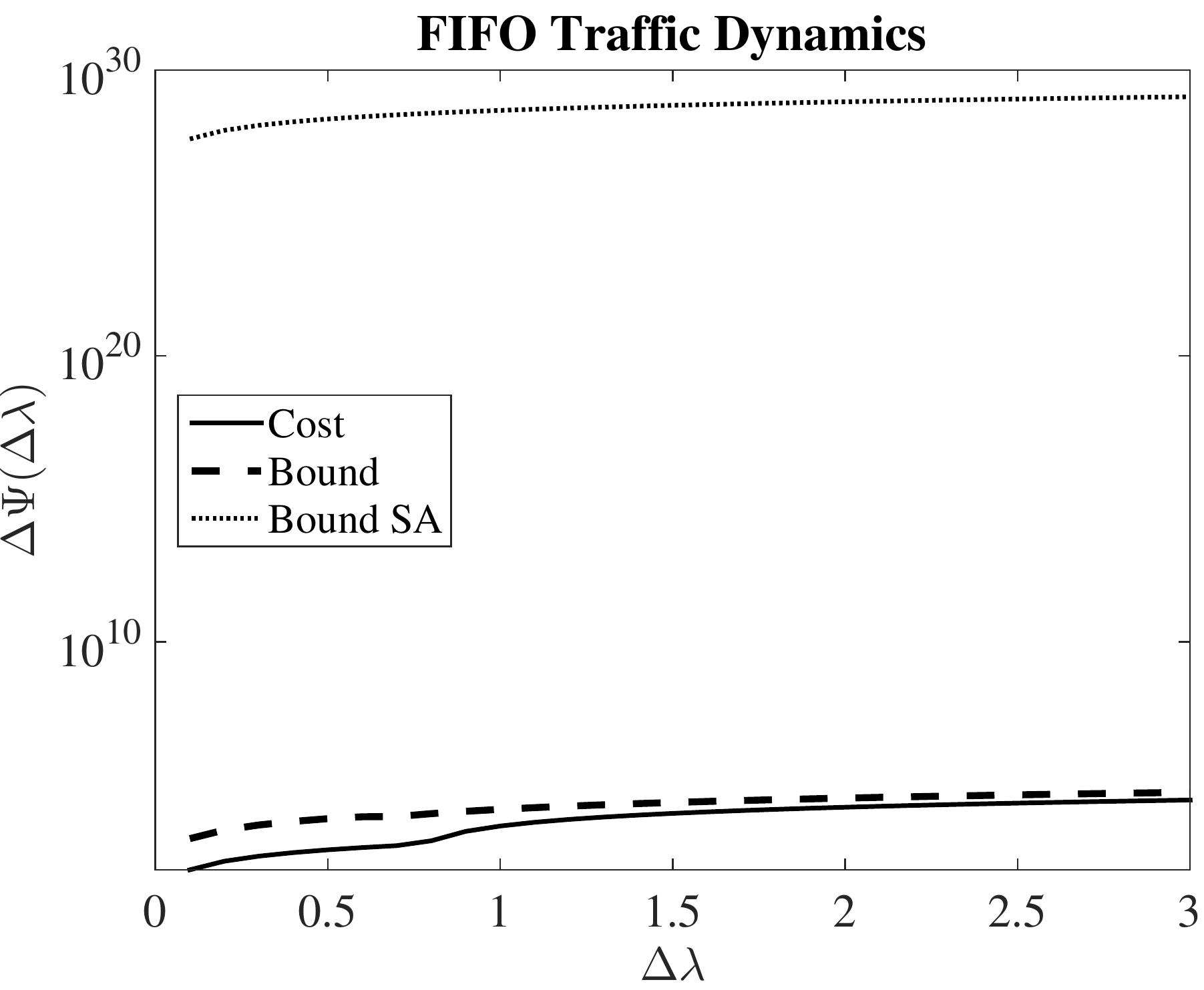}
\end{tabular}
\caption{Comparison of the perturbation in cost due to perturbation in inflow under FIFO traffic dynamics obtained from simulations (solid line), corresponding bounds from Propositions~\ref{theo:sensitivity} and \ref{thm:robustness}, \eqref{eq:triangle-ineq} and \eqref{eq:growth-approx} (dashed lines), and (on the right panel) from Sensitivity analysis \eqref{eq:sensitivity-bound-x-constant-delta-lambda}.}
\label{figure:cost_function_of_lambda_constantinflow_FIFO}
\end{figure}

\begin{figure}
\centering
\begin{tabular}{cc}
\includegraphics[height=4.5 cm]{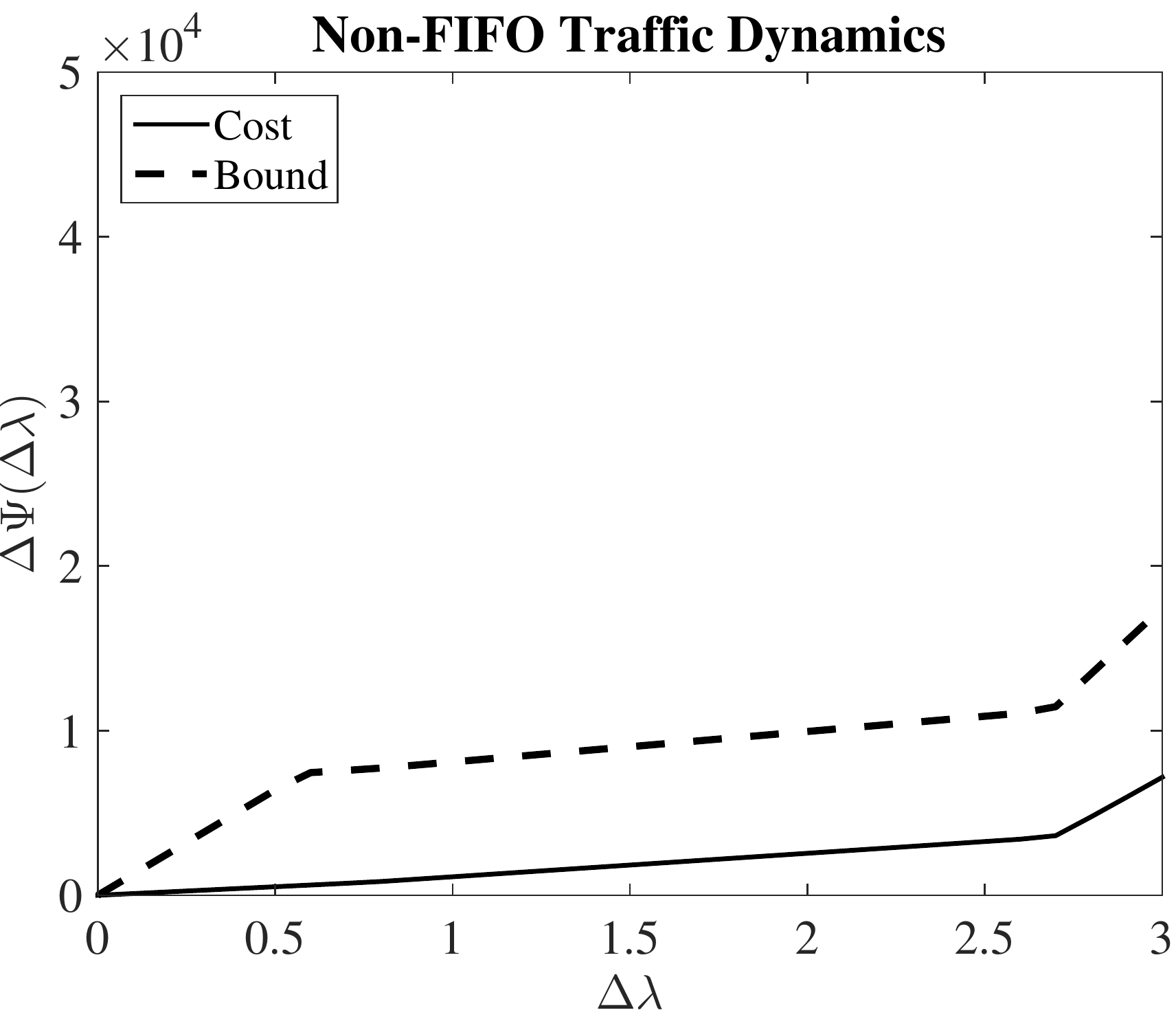}
&
\includegraphics[height=4.5 cm]{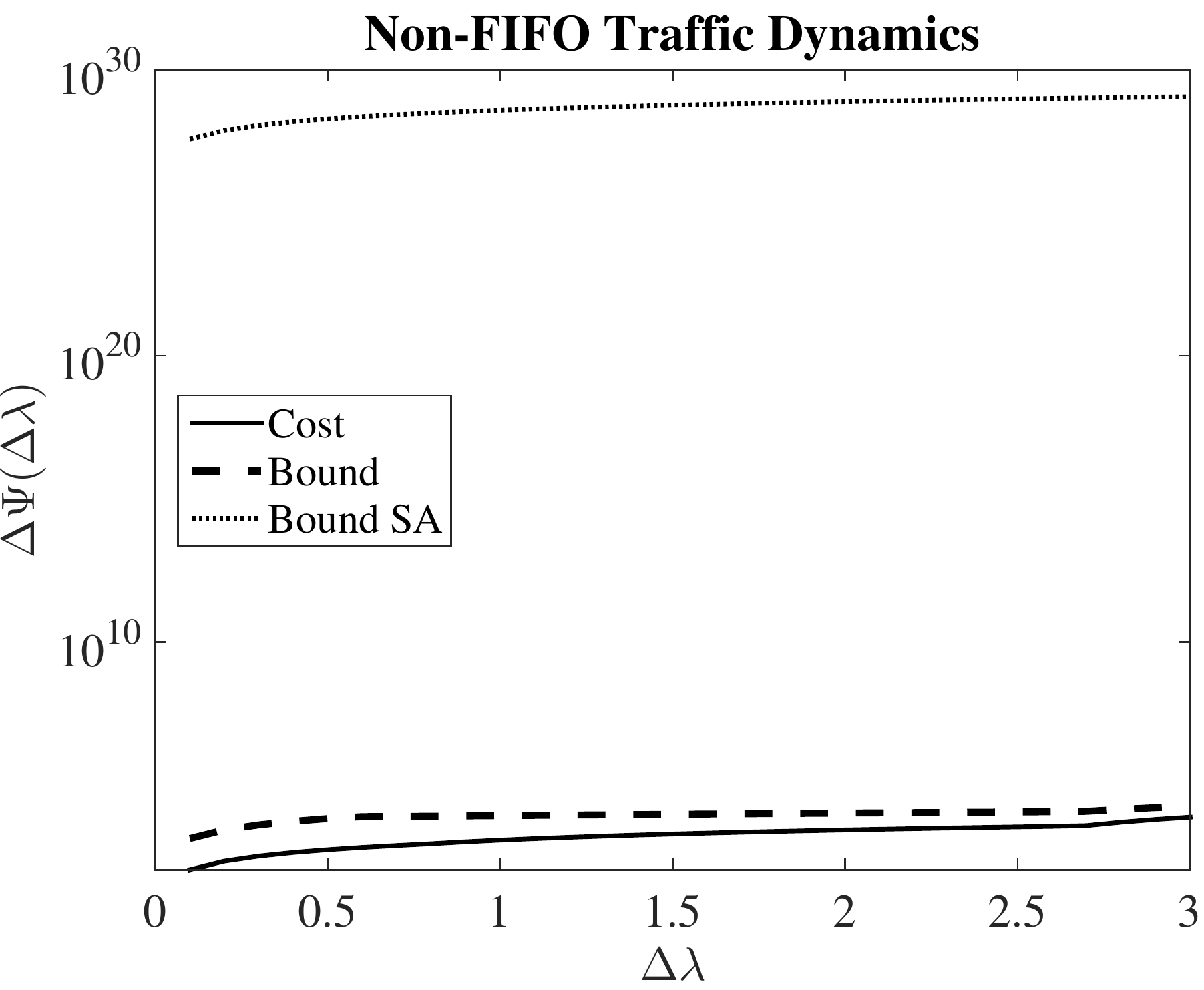}
\end{tabular}
\caption{Comparison of the perturbation in cost due to perturbation in inflow under non-FIFO traffic dynamics obtained from simulations (solid line), corresponding bounds from Propositions~\ref{theo:sensitivity} and \ref{thm:robustness}, \eqref{eq:triangle-ineq} and \eqref{eq:growth-approx} (dashed lines), and (on the right panel) from Sensitivity analysis \eqref{eq:sensitivity-bound-x-constant-delta-lambda}.}
\label{figure:cost_function_of_lambda_constantinflow_nonFIFO}
\end{figure}

The solid line plots for both FIFO and non-FIFO traffic dynamics in are piece-wise affine, with the transition point corresponding to network capacity. This transition point is at $\Delta\lambda \approx 0.8$ and $\Delta\lambda \approx 2.8$ for the FIFO and non-FIFO dynamics, respectively. In fact, when $\Delta\lambda$ is less than the transition point, the trajectories $\tilde{x}^{(\Delta \lambda)}(t)$ are found to reach equilibrium, thereby allowing application of Proposition~\ref{thm:robustness}. Moreover, in this regime, the trajectories are in the free-flow region. 
When $\Delta\lambda$ exceeds the transition point, then the trajectories are in overload, and the traffic volumes on onramps grow unbounded, which results in a higher slope in the perturbation plot in both Figure~\ref{figure:cost_function_of_lambda_constantinflow_FIFO} and Figure~\ref{figure:cost_function_of_lambda_constantinflow_nonFIFO}. For $\Delta\lambda \leq 0.8$ (corresponding to the transition point for the FIFO traffic dynamcis), the perturbed trajectories under both dynamics are identical. That is, the better robustness property of non-FIFO models can be attributed to a higher network capacity, which in turn can be attributed to flexibility in routing of traffic flow in congestion. These features are well captured by the analytical bounds computed by Proposition~\ref{theo:sensitivity} and \ref{thm:robustness} (when an equilibrium exists) and \eqref{eq:triangle-ineq} and \eqref{eq:growth-approx} (in overload), which are shown as dashed lines in Figure~\ref{figure:cost_function_of_lambda_constantinflow_FIFO} and Figure~\ref{figure:cost_function_of_lambda_constantinflow_nonFIFO}. In general, the perturbed trajectories under controls derived from DTA or FNC optimal solutions show qualitatively similar behavior, with the difference being in the locations of transition points.    

Finally, the bound resulting from Sensitivity anaylsis and \eqref{eq:sensitivity-bound-x-constant-delta-lambda} is shown only on the right panel and in log-scale. Differently from the robustness bounds that we derive, the order of magnitude of the error provided by Sensitivity analysis is not correct, providing over-conservative upper bounds. The reason is that while our robustness bounds depend on the time horizon $T$ in a linear way, the Sensitivity analysis provides a worst-case bound which grows exponentially with $T$. 

\subsection{Robustness-Performance Tradeoff under Reduced Feasible Set for DTA}
\label{sec:simulations-congestion}

From the analysis proposed in Section~\ref{section:robustness} is can be noticed that the network controlled via Proposition~\ref{proposition:implementation} exhibits less robustness to inflow perturbations when the perturbed trajectories are in the congestion region. In this last numerical study we investigate the possibility of scaling down the supply functions in the DTA formulations, whose optimal solution is used to set the controls. Such a modification introduces more \emph{slack}, thereby increasing the magnitude of inflow perturbations beyond which the trajectories enter the congestion region, and as such increases the network robustness. Naturally, such a modification comes at the expense of increase in the value of optimal cost under zero perturbation. The objective in this section is to study this tradeoff.
      
Specifically, we replace the first inequality in \eqref{supply-demand-const} with
	\begin{equation}
		\label{eq:supplyRobust}
			y_i(t) \leq s_i(x_i)(1- \varepsilon)
		\tag{\ref{supply-demand-const}'}
	\end{equation}
where $0\leq \varepsilon\leq1$ is a tunable parameter. In this section, we focus on the controls computed from optimal solutions to FNC; the results are qualitatively similar for DTA.

We consider again the total traffic volume cost $\psi(x) = \sum_{i\in\mc E}x_i$ and the scenario described in Section~\ref{sec:simulations-implementation}. For a given $\epsilon \geq 0$, similar to Section~\ref{section:robustness}, we computed the control values from the optimal solution to FNC with initial traffic volume $x(0)=0$, and then simulated the system 
with several values of perturbations to inflow $\Delta \lambda$.
The simulations were repeated for a range of values of the parameter $\varepsilon$.  Let $\tilde{x}^{(\varepsilon, \Delta\lambda)}(t)$ denote the perturbed trajectory. For each combination of $\varepsilon$ and $\Delta \lambda$, we computed the perturbed cost: $\Psi(\varepsilon, \Delta\lambda) := \sum_{t=0}^{25}\sum_{i\in\mc E} \tilde{x}_i^{(\varepsilon, \Delta\lambda)}(t)$. Due to the relatively short time horizon setup of this scenario, the perturbed cost does not necessarily reflect the congestion effects in the perturbed system. Therefore, we additionally record congestion factor: $\gamma(\varepsilon, \Delta\lambda) = \min_{j,t}\gamma_j(\tilde{x}^{(\varepsilon,\Delta\lambda)}(t)$, where the minimum is taken over $v\in\V$ and all times, and where $\gamma_j$ is defined in \eqref{eq:gammaFIFO}. 
Clearly, $\gamma(\varepsilon, \Delta\lambda)\leq 1$, where equality implies that the perturbed trajectory is in free-flow all the time, whereas strict inequality implies congestion. The smaller the value of $\gamma(\varepsilon, \Delta\lambda)$ is, the more congested some cells are.
	

The results are shown in Fig.~\ref{figure:influenceConstrainAndPerturbation}. As expected, for every $\Delta \lambda$, the perturbed cost is an increasing function of $\varepsilon$. Concerning congestion factor, $\gamma(\varepsilon, 0) = 1$ for all $\varepsilon$, as the optimal trajectory is always in free-flow, while for high $\Delta\lambda$, the lower $\varepsilon$ is, the more the system is prone to congestion. At the two extremes, for $\varepsilon = 0$ a small inflow increase yields congestion; for $\varepsilon=0.5$, the perturbed solutions are in free-flow even for relatively high values of $\Delta\lambda$. 
		
\def \WIDTHSINGLE {5}
\def \HEIGHTSINGLE {4}

\begin{figure}[htb!]
\centering
\begin{tabular}{cc}
\includegraphics[width=\WIDTHSINGLE cm, height=\HEIGHTSINGLE cm]{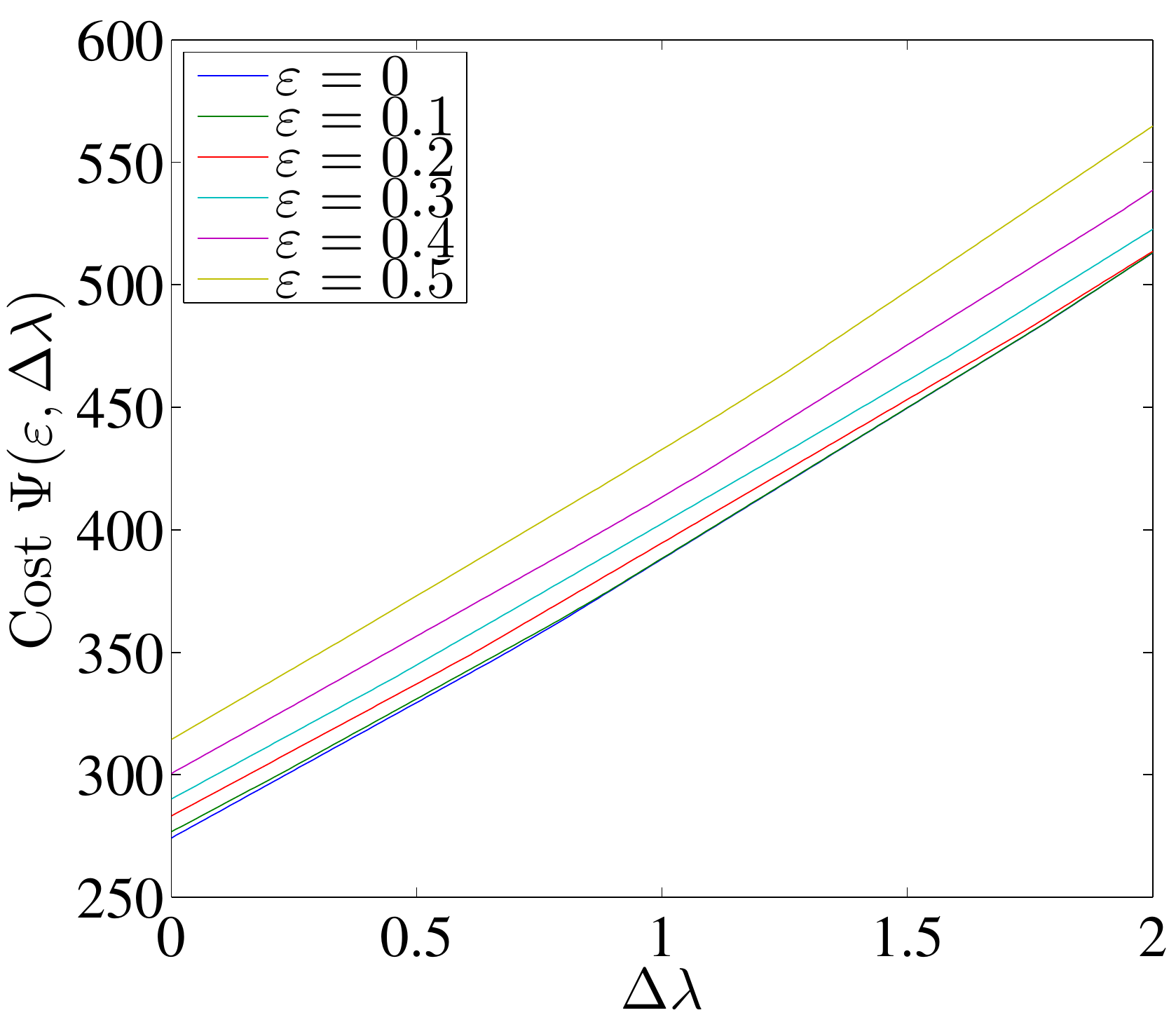}&
\includegraphics[width=\WIDTHSINGLE cm, height=\HEIGHTSINGLE cm]{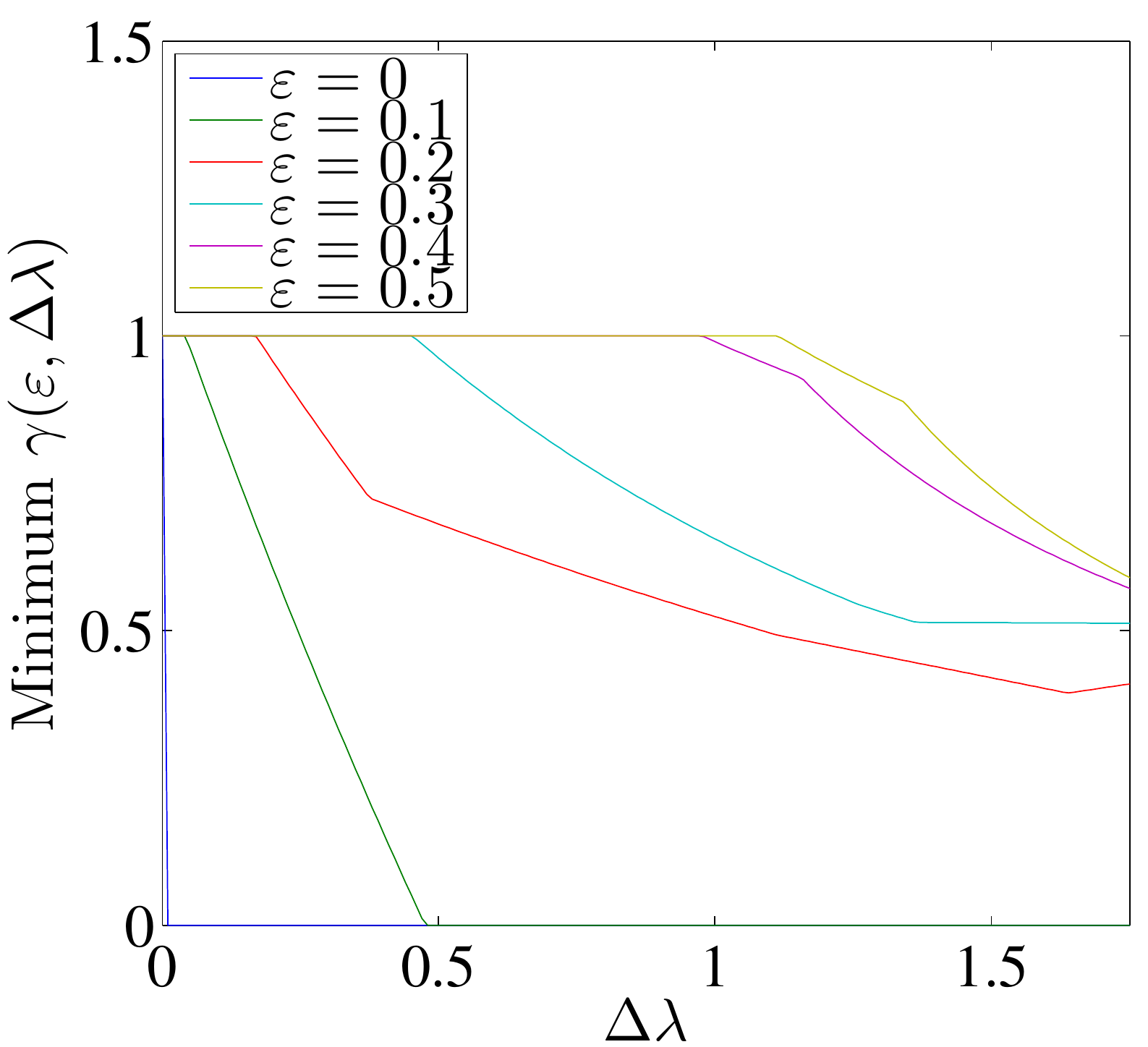}\\
\end{tabular}
\caption{Impact of reduced feasible set and inflow perturbation. Left panel: cost $\Psi(\varepsilon, \Delta\lambda)$ as a function of $\Delta\lambda$ for different $\varepsilon$ (colored lines). Right panel: congestion factor $\gamma(\varepsilon, \Delta\lambda)$ as a function of $\Delta\lambda$ for different $\varepsilon$.}
\label{figure:influenceConstrainAndPerturbation}
\end{figure}

\section{Conclusion}
\label{section:conclusion}

We consider two variants of cell-based continuous time System Optimum Dynamic Traffic Assignment, which we refer to as DTA problem, in which turning ratios are a possible control input, and FNC problem, in which turning ratios are exogenously assigned. Within these formulations we study relaxations of traffic dynamics, where total inflow into and total outflow from the cells are upper bounded by supply and demand respectively, and which are shown to yield convex problems. We also design open-loop variable speed limits, ramp metering and routing controllers that ensure feasibility of the optimal solutions. This significantly expands known results in terms of relationship between computationally efficient DTA formulations and feasibility of their optimal solutions with respect to traffic dynamics. Furthermore, using monotonicity arguments, we derive bounds on the cost increase due to perturbations of initial conditions and external inflows under the designed open loop controllers. The proposed methodologies are illustrated via extensive simulation results. 

Future research direction includes leveraging necessary and sufficient conditions along the lines of Section~\ref{sec:optimal-control-necessary-condition} and \cite{Seierstad.Sydsaeter:77,Friesz.Luque.ea:89} to possibly speed up numerical solutions to various variants of DTA. We also plan to develop distributed algorithms for solving the DTA and FNC problems, along the lines of our preliminary work in \cite{Ba.Savla.ea:CDC15}.

\section*{Acknowledgements}
G.~Como was partially supported by the Swedish Research Council through the Junior Research Grant Information Dynamics in Large Scale Networks and the Linnaeus Excellence Center, LCCC. K.~Savla was supported in part by METRANS Research Initiation Award 14-09 and NSF ECCS Grant No.~1454729. The authors are grateful to Prof.~Anders Rantzer for his many useful comments and encouragement during this research.

\bibliographystyle{plain}
\bibliography{bibliography} 

\begin{thebibliography}{10}

\bibitem{Ba.Savla.ea:CDC15}
Q.~Ba, K.~Savla, and G.~Como.
\newblock Distributed optimal equilibrium selection for traffic flow over
  networks.
\newblock In {\em IEEE Conference on Decision and Control}, pages 6942--6947,
  2015.

\bibitem{Como.Lovisari.ea:TRb16-extended}
G.~Como, E.~Lovisari, and K.~Savla.
\newblock Convexity and robustness of dynamic traffic assignment for control of
  freeway networks (extended version).
\newblock Available at \texttt{http://arxiv.org/abs/1509.06189}, 2015.

\bibitem{Como.Lovisari.ea:TCNS15}
G.~Como, E.~Lovisari, and K.~Savla.
\newblock Throughput optimality and overload behavior of dynamical flow
  networks under monotone distributed routing.
\newblock {\em Control of Network Systems, IEEE Transactions on}, 2(1):57--67,
  March 2015.

\bibitem{ComoPartITAC13}
G.~Como, K.~Savla, D.~Acemoglu, M.~A. Dahleh, and E.~Frazzoli.
\newblock Robust distributed routing in dynamical networks - part i: Locally
  responsive policies and weak resilience.
\newblock {\em IEEE Transactions on Automatic Control}, 58(2), 2013.

\bibitem{ComoPartIITAC13}
G.~Como, K.~Savla, D.~Acemoglu, M.~A. Dahleh, and E.~Frazzoli.
\newblock Robust distributed routing in dynamical networks - part ii: Strong
  resilience, equilibrium selection and cascaded failures.
\newblock {\em IEEE Transactions on Automatic Control}, 58(2), 2013.

\bibitem{CooganACC14}
S.~Coogan and M.~Arcak.
\newblock Dynamical properties of a compartmental model for traffic networks.
\newblock In {\em Proceedings of the IEEE American Control Conference (ACC)},
  pages 2511 -- 2516, 2014.

\bibitem{CVXSoftware}
Inc. CVX~Research.
\newblock {CVX}: Matlab software for disciplined convex programming, version
  2.0.
\newblock \url{http://cvxr.com/cvx}, August 2012.

\bibitem{Daganzo:94}
C.~F. Daganzo.
\newblock The cell transmission model: {A} dynamic representation of highway
  traffic consistent with the hydrodynamic theory.
\newblock {\em Transportation Research B: Methodological}, 28B(4):269--287,
  1994.

\bibitem{Daganzo:95}
C.~F. Daganzo.
\newblock The cell transmission model, part {II}: network traffic.
\newblock {\em Transportation Research B: Methodological}, 29B(2):79--93, 1995.

\bibitem{DoanTRB:12}
K.~Doan and S.V. Ukkusuri.
\newblock On the holding-back problem in the cell transmission based dynamic
  traffic assignment models.
\newblock {\em Transportation Research Part B: Methodological}, 46(9):1218 --
  1238, 2012.

\bibitem{Friesz.Luque.ea:89}
T.L. Friesz, J.~Luque, R.L. Tobin, and B.-W. Wie.
\newblock Dynamic network traffic assignment considered as a continuous time
  optimal control problem.
\newblock {\em Operations Research}, 37(6):893--901, 1989.

\bibitem{GomesTRC06}
G.~Gomes and R.~Horowitz.
\newblock Optimal freeway ramp metering using the asymmetric cell transmission
  model.
\newblock {\em Transportation Research Part C}, 14(4):244--268, 2006.

\bibitem{GrantRALC08}
M.~Grant and S.~Boyd.
\newblock Graph implementations for nonsmooth convex programs.
\newblock In V.~Blondel, S.~Boyd, and H.~Kimura, editors, {\em Recent Advances
  in Learning and Control}, Lecture Notes in Control and Information Sciences,
  pages 95--110. Springer-Verlag Limited, 2008.

\bibitem{HegyiTRC05}
A.~Hegyi, B.~De~Schutter, and H.~Hellendoorn.
\newblock Model predictive control for optimal coordination of ramp metering
  and variable speed limits.
\newblock {\em Transportation Research Part C}, 13(3):185 -- 209, 2005.

\bibitem{HirschPS03}
M.~Hirsch and H.L. Smith.
\newblock Competitive and cooperative systems: A mini-review.
\newblock {\em Positive Systems. Lecture Notes in Control and Information
  Sciences}, 294, 2003.

\bibitem{Jabari.Liu:12}
S.~E. Jabari and H.~X. Liu.
\newblock A stochastic model of traffic flow: Theoretical foundations.
\newblock {\em Transportation Research Part B: Methodological}, 46(1):156--174,
  2012.

\bibitem{Karafyllis.Papageorgiou:TCNS14}
I.~Karafyllis and M.~Papageorgiou.
\newblock Global exponential stability for discrete-time networks with
  applications to traffic networks.
\newblock {\em IEEE Transactions on Control of Network Systems}, 2(1):68--77,
  March 2015.

\bibitem{LighthillTrafficPTRS55}
M.~J. Lighthill and G.~B. Whitham.
\newblock On kinematic waves. ii. a theory of traffic flow on long crowded
  roads.
\newblock {\em Phil. Trans. R. Soc. A}, 229(1178):317--345, 1955.

\bibitem{Lovisari.Como.ea:CDC14}
E.~Lovisari, G.~Como, and K.~Savla.
\newblock Stability of monotone dynamical flow networks.
\newblock In {\em Proceedings of the 53rd IEEE Conference on Decision and
  Control}, pages 2384 -- 2389, Los Angeles, CA, 2014.

\bibitem{MerchantTSa:78}
D.K. Merchant and G.L. Nemhauser.
\newblock A model and an algorithm for the dynamic traffic assignment problem.
\newblock {\em Transportation Science}, 12:183--199, 1978.

\bibitem{MerchantTSb:78}
D.K. Merchant and G.L. Nemhauser.
\newblock Optimality conditions for a dynamic traffic assignment model.
\newblock {\em Transportation Science}, 12:200--207, 1978.

\bibitem{MuralidharanACC12}
A.~Muralidharan and R.~Horowitz.
\newblock Optimal control of freeway networks based on the link node cell
  transmission model.
\newblock In {\em Proceedings of the American Control Conference (ACC)}, pages
  5769--5774, June 2012.

\bibitem{PeetaNSE:01}
S.~Peeta and A.~K. Ziliaskopoulos.
\newblock Foundations of dynamic traffic assignment: The past, the present and
  the future.
\newblock {\em Networks and Spatial Economics}, 1(3-4):233--265, 2001.

\bibitem{Richards:56}
I.~Richards.
\newblock Shockwaves on the highway.
\newblock {\em Operations Research}, 4:42--51, 1956.

\bibitem{Seierstad.Sydsaeter:77}
A.~Seierstad and K.~Sydsaeter.
\newblock Sufficient conditions in optimal control theory.
\newblock {\em International Economic Review}, pages 367--391, 1977.

\bibitem{Waller.Ziliaskopoulos:06}
S.~T. Waller and A.K. Ziliaskopoulos.
\newblock A chance-constrained based stochastic dynamic traffic assignment
  model: Analysis, formulation and solution algorithms.
\newblock {\em Transportation Research Part C: Emerging Technologies},
  14(6):418--427, 2006.

\bibitem{Ziliaskopoulos:00}
A.K. Ziliaskopoulos.
\newblock A linear programming model for the single destination system optimum
  dynamic traffic assignment problem.
\newblock {\em Transportation science}, 34(1):37--49, 2000.

\end{thebibliography}

\appendix
\section{Proofs}
\subsection{Proof of Proposition \ref{proposition:implementation}}
\label{proof:proposition:implementation}
\begin{enumerate}\item[(i)] 
Let $(x(t),y(t),z(t),\mu(t),f(t))$ be a feasible solution of the convex optimal control problem \eqref{DTA-1}.
At any time $t\ge0$, the demand constraints in \eqref{supply-demand-const} imply that $z_i(t)\leq C_i$. It follows that, with the choice of demand control parameters in \eqref{eq:controlVariablesFC_alpha}, for  every non-source cell $i\in\E\setminus\R$,  
	$$
		z_i(t) = \min\{z_i(t), C_i\} = \min\{\alpha_i(t) d_i(x_i(t)), C_i\}=\ov d_i(x_i(t),\alpha_i(t))\,.
	$$ 
For sink cells $k\in\mc S$, the above readily implies that $\mu_k(t)=z_k(t)=\ov d_k(x_k(t),\alpha_k(t))$, so that \eqref{DNL1} holds true. On the other hand, for every source cell $i\in\R$, the choice of demand control parameters in \eqref{eq:controlVariablesFC_alpha} and the demand constraint $z_i(t)\le d_i(t)$ in \eqref{supply-demand-const} imply that  
	$$
	z_i(t)=\min\{ d_i(x_i(t)),z_i(t)\}=\min\{d_i(x_i(t)),\alpha_i(t)C_i\}=\ov d_i(x_i(t),\alpha_i(t))\,.		$$ 
Then, $z_i(t)=\ov d_i(x_i(t),\alpha_i(t))$ for every cell $i\in\mc E$ and the choice of the routing matrix \eqref{eq:controlVariablesFC_R} implies that 
	$$
		f_{ij}(t)
			=	R_{ij}(t)z_i(t)
			=	R_{ij}(t)\ov d_i(x_i(t),\alpha_i(t))\,,
	$$
	for every pair of adjacent cells $(i,j) \in \A$. 
Therefore, for every cell $j\in\mc E$, 
	\begin{equation}
	\label{eq:freeflow}
	\sum_{i\in\E}R_{ij}(t)\ov d_i(x_i(t),\alpha_i(t))	= \sum_{i:(i,j)\in\A}f_{ij}(t)=y_j(t) \leq s_j(x_j(t))\,,
	\end{equation}
where the inequality is implied by the supply constraint in \eqref{supply-demand-const}. It follows from \eqref{eq:gammaFIFO} that $\gamma_i^F(t)=1$ for every cell $i\in\mc E$ and this implies that 
	\be
		\label{fijFN}
\gamma_i^F(t)R_{ij}(t) \ov d_i(x_i(t),\alpha_i(t)) = f_{ij}(t)\,,\qquad (i,j)\in\mathcal{A}\,,
	\ee
so that \eqref{eq:fjFIFO} is satisfied. Finally, the choice of the controlled routing matrix \eqref{eq:controlVariablesFC_R} can be readily verified to satisfy \eqref{R-const} and \eqref{sumR=1}. 
Hence, for every feasible solution $(x(t),y(t),z(t),\mu(t),f(t))$ of the convex optimal control problem \eqref{DTA-1}, the choices \eqref{eq:controlVariablesFC_alpha} and \eqref{eq:controlVariablesFC_R} of the demand control parameters $\alpha(t)$ and of the 
controlled routing matrix $R(t)$ satisfy the constraints \eqref{R-const}-\eqref{eq:gammaFIFO}, so that $(\alpha(t),R(t))$ is a feasible solution of the DTA problem \eqref{DTA-0}.
\item[(ii)] 
Let $(x(t),y(t),z(t),\mu(t),f(t))$ be a feasible solution of the convex optimal control problem \eqref{FNC-1}.
Then, for every time $t\ge0$, the choice of the demand control parameters \eqref{eq:controlVariablesFC_alpha} along with the additional constraint \eqref{additional-const-2} imply that 
	$$
		f_{ij}(t)= R_{ij}(t) z_i(t) = R_{ij}(t) \ov d_i(x_i(t),\alpha_i(t))\,,
	$$ 
 for every pair of adjacent cells $(i, j) \in \A$. 
It follows that
	\begin{align*}
		\sum_{i\in\E} R_{ij}\ov d_i(x_i(t),\alpha_i(t))	= \sum_{i: (i,j)\in\A}f_{ij}(t)=y_i(t) \leq s_j(x_j(t))\,,\qquad j \in \E\,,
	\end{align*}
where the inequality is implied by the supply constraint in \eqref{supply-demand-const}. From this, \eqref{fijFN} follows as in (i), so that \eqref{eq:fjFIFO} is satisfied. 
On the other hand, one has that $\mu_k(t)=z_k(t)=\ov d_k(x_k(t),\alpha_k(t))$ for every sink cell $k\in\mc S$, so that also \eqref{DNL1} is satisfied. 
Hence, for every feasible solution $(x(t),y(t),z(t),\mu(t),f(t))$ of the convex optimal control problem \eqref{FNC-1}, the choice \eqref{eq:controlVariablesFC_alpha}  of the demand control parameters $\alpha(t)$ satisfie the constraints \eqref{DNL1}-\eqref{eq:gammaFIFO}, so that $\alpha(t)$ is a feasible solution of the FNC problem \eqref{FNC-0}.

\end{enumerate}
The fact that $x(t)$ is in free-flow under the designed controls follows from \eqref{eq:freeflow} in both cases (i) and (ii).
\hfill$\square$
\medskip

\subsection{Proof of Proposition \ref{prop:optimal-control-FNC}}

If $v=\sigma_i$ is a merging junction with downstream cell $j$, then we immediately get 
that 
$$(z^*_i(t))_{i:\,\tau_i=v}\in\argmax_{\substack{\\[5pt]\ds 0\le z_i\le C_i\\[5pt]\ds z_i\le d_i(x^*_i(t))\\[5pt]\ds \sum_{\substack{i\in\mc E:\\\tau_i=v}} z_i\le s_j(x^*_j(t))}}\sum_{\substack{i\in\mc E:\\\tau_i=v}}\kappa_i(t)z_i(t)\,.$$

On the other hand, if $\tau_i$ is either an ordinary or a diverging junction, then a minimizer in \eqref{LP-dual} necessarily satisfies 
\be\label{xiR} \chi^*_i(t)=0\,,\qquad \xi^*_i(t)+\sum_{j\in\mc E}R_{ij}\nu^*_j(t)=\max\{0,\kappa_i(t)\}\,.\ee
Indeed, reducing $\chi_i\ge0$, $\xi_i\ge0$, and $\nu_j\ge0$ for all $j$ such that $\sigma_j=\tau_i$ until $\chi_i=0$ and $\xi_i+\sum_{j\in\mc E}R_{ij}\nu_j=\max\{0,\kappa_i\}$ reduces the cost 				$$
		\sum_{i\in\mc E}\left(\xi_id_i(x^*_i)+\nu_is_i(x^*_i)+\chi_iC_i\right)
	$$ 
without violating any of the constraints in \eqref{LP-dual}.
(The fact that $\tau_i$ is either an ordinary or a diverging junction implies that such variables $\nu_j$ appear in no constraint $\chi_h+\xi_h+\sum_{j\in\mc E}R_{hj}\nu_j\ge\kappa_h$ other than for $h=i$.)
Then, by combining \eqref{taudot} and \eqref{xiR} we get that 
	\be
	\label{tau-Rtau}
	\ba{rcl}
	\ds\dot\kappa_i&=&\ds\dot\zeta_i-\sum_{j\in\mc E}R_{ij}\dot\zeta_j\\[10pt]
		&=&\ds-1+\sum_{j\in\mc E}R_{ij} +\omega \left(-\nu^*_i+\sum_{j\in\mc E}R_{ij}\nu^*_j+\xi_i^*-\sum_{j\in\mc E}R_{ij}\xi^*_j\right)\\[10pt]
		&\le&\ds \omega \max\{0,\kappa_i\}\,,
	\ea
	\ee
where the last step follows from the constraints $\sum_{j\in\mc E}R_{ij}=1$, $\nu_i\ge0$, and $\sum_{j\in\mc E}R_{ij}\xi_j\ge0$.
It then follows from \eqref{tau-Rtau} and \eqref{tauT} that 
\be \kappa_i(t)\ge0\,,\qquad t\in[0,T]\,.\ee
Indeed, if $\kappa_i(t)<0$ for some $t\in[0,T]$, then \eqref{tau-Rtau} implies that $\dot \kappa_i(t)=0$, hence  $\kappa_i(t')=\kappa_i(t)<0$ for all $t'\ge t$, which would contradict \eqref{tauT}.
It follows that
	\be
		\label{ordinary-diverging}
		z^*_i(t)=\argmax_{\substack{\\[5pt]\ds 0\le z_i\le C_i \\[5pt]\ds z_i\le d_i(x^*_i(t))\\[5pt]\ds R_{ij}z_i\le s_j(x^*_j(t))}}\kappa_iz_i =
\max_{\substack{\\[5pt]\ds 0\le z_i\le C_i \\[5pt]\ds z_i\le d_i(x^*_i(t))\\[5pt]\ds R_{ij}z_i\le s_j(x^*_j(t))}}z_i\,.\ee
If $\tau_i$ is an ordinary junction, then \eqref{ordinary-diverging} reduces to \eqref{ordinary}.
If $\tau_i$ is a diverging junction, then it coincides with \eqref{diverge}.




\subsection{Proof of Proposition~\ref{theo:sensitivity}}
The proof is based on application of contraction principles, developed in our previous work~\cite{Como.Lovisari.ea:TCNS15}, for system trajectories in the monotone region. 
First note that
\begin{align}
	\frac{d}{dt}||\tilde{x}(t) - x(t)||_1
		=	& \sum_{i\in\E}\sgn{\tilde{x}_i(t) - x_i(t)}(\tilde{\lambda}_i(t) - \lambda_i(t)) \nonumber \\
			& \qquad+ \sum_{i\in\E}\sgn{\tilde{x}_i(t) - x_i(t)}(g_i( \tilde x,\alpha,R) - g_i(x, \alpha,R) \nonumber \\
	\leq 	& ||\tilde{\lambda}(t)-\lambda(t)||_{1} \nonumber\\
	\label{eq:contraction-last-ineq}
			& \qquad+ \sum_{i\in\E}\sgn{\tilde{x}_i(t) - x_i(t)}(g_i(x,\alpha,R) - g_i(\tilde{x},\alpha,R))
\end{align}

For sufficiently small perturbations, $x(t)$ and $\tilde{x}(t)$ remain in the free-flow region for all $t \in [0,T]$. Therefore, the traffic flow dynamics is monotone, hence 
\cite[Lemma 1]{Como.Lovisari.ea:TCNS15} implies that the second term in the RHS of \eqref{eq:contraction-last-ineq} is non-positive. Therefore,
\vspace{-.1cm}
	\begin{align*}
		\frac{d}{dt}||\tilde{x}(t) - x(t)||_1 \leq ||\tilde{\lambda}(t)-\lambda(t)||_{1} \qquad \forall \, t \in [0,T]
	\end{align*}
which, upon integration, gives \eqref{eq:nonFIFOResultTraj}.

\subsection{Proof of Proposition~\ref{thm:robustness}}
Let $\phi(t,x^0,\lambda)$ denote the state of the network at time $t$ starting from initial traffic volume $x^0$ and under inflow $\lambda$. Assume that the perturbation is small enough so that the system is monotone under perturbation. Monotonicity implies that, for all $t\in [0,T]$ and all $i\in\E$, it holds true
	\begin{align*}
		\phi_i(t, \bar{x}^0, \bar{\lambda})
		& \geq
		\phi_i(t, x^0, \lambda)
		 \geq 
		\phi_i(t, \underline{x}^0, \underline{\lambda})\\
		\phi_i(t, \bar{x}^0, \bar{\lambda})
		& \geq
		\phi_i(t, \tilde{x}^0, \tilde{\lambda})
		 \geq 
		\phi_i(t, \underline{x}^0, \underline{\lambda}) 
	\end{align*}
so
	\begin{equation}
	\label{eq:monotonicity-implication}
		|\phi_i(t, x^0, \lambda) - \phi_i(t, \tilde{x}^0, \tilde{\lambda})|
		\leq
		\phi_i(t, \bar{x}^0, \bar{\lambda}) - \phi_i(t, \underline{x}^0, \underline{\lambda})
	\end{equation}
In addition, \cite[Lemma 1]{Como.Lovisari.ea:TCNS15} implies that
	\begin{equation}
	\label{eq:contraction}
		||\phi(t, x, \lambda) - \phi(t, y, \lambda)||_1
		\leq
		||x-y||_1, \, \, \forall t\in[0,T]
	\end{equation}
	
Therefore,
	\begin{equation}
	\label{eq:ineq1}
	\begin{split}
		&||\phi(t, x^0, \lambda) - \phi(t, \tilde{x}^0, \tilde{\lambda})||_1  \\
			&	\quad\leq ||\phi(t, \bar{x}^0, \bar{\lambda}) - \phi(t, \underline{x}^0, \underline{\lambda})||_1  \\
			&	\qquad\leq
			||\phi(t, \bar{x}^0, \bar{\lambda}) - \phi(t, \bar{x}^0, \underline{\lambda})||_1+||\phi(t, \bar{x}^0, \underline{\lambda}) - \phi(t, \underline{x}^0, \underline{\lambda})||_1  \\
			&	\qquad\leq
			||\phi(t, \bar{x}^0, \bar{\lambda}) - x^{\text{eq}}(\bar{\lambda})||_1 + ||x^{\text{eq}}(\bar{\lambda}) - x^{\text{eq}}(\underline{\lambda})||_1 \\
			&\qquad\qquad\qquad+||x^{\text{eq}}(\underline{\lambda})- \phi(t, \bar{x}^0, \underline{\lambda})||_1+ 
			||\phi(t, \bar{x}^0, \underline{\lambda}) - \phi(t, \underline{x}^0, \underline{\lambda})||_1  \\
			&	\quad\leq
			||\bar{x}^0 - x^{\text{eq}}(\bar{\lambda})||_1 + ||x^{\text{eq}}(\bar{\lambda}) - x^{\text{eq}}(\underline{\lambda})||_1 +
			||x^{\text{eq}}(\underline{\lambda})- \bar{x}^0||_1+ 
			||\bar{x}^0 - \underline{x}^0||_1
	\end{split}
	\end{equation}	

where the first inequality follows from \eqref{eq:monotonicity-implication}, the second and third by triangle inequality, and the fourth by \eqref{eq:contraction}. Also, by exchanging the two terms of the difference after the first inequality in \eqref{eq:ineq1}, we have
	\begin{equation}
	\label{eq:ineq2}
	\begin{split}
		||\phi(t, x^0, \lambda) - \phi(t, \tilde{x}^0, \tilde{\lambda})||_1
			&	\leq
			||\underline{x}^0 - x^{\text{eq}}(\underline{\lambda})||_1 + ||x^{\text{eq}}(\underline{\lambda})-x^{\text{eq}}(\bar{\lambda})||_1\\
			&	\qquad + ||x^{\text{eq}}(\bar{\lambda})- \underline{x}_0||_1 +  ||\underline{x}^0-\bar{x}^0||_1\,.
	\end{split}	
	\end{equation}
The result follows by combining \eqref{eq:ineq1} and \eqref{eq:ineq2}.

\subsection{Robustness Bounds from Sensitivity Analysis}
\label{sec:sensitivity-analysis}
\eqref{eq:dynamics-succinct} implies that 
$$
\frac{d}{dt} (\tilde{x}_i(t)-x_i(t)) = (\tilde{\lambda}_i(t)-\lambda_i(t)) - \left(g_i(\tilde{x},\alpha,R) - g_i(x,\alpha,R) \right), \qquad i \in \mc E
$$
which implies that
$$
\left| \frac{d}{dt}  (\tilde{x}_i(t)-x_i(t)) \right| \leq |\tilde{\lambda}_i(t)-\lambda_i(t)| +  \left|g_i(\tilde{x},\alpha,R) - g_i(x,\alpha,R) \right|
$$
%
Summing over $i \in \mc E$, and utilizing \eqref{eq:global-lipschitz}, we get that
\begin{equation*}
\sum_{i \in \mc E} \left| \frac{d}{dt} (\tilde{x}_i(t)-x_i(t)) \right| \leq \|\tilde{\lambda}(t)-\lambda(t)\|_1 + L_g \|(\tilde{x}(t)-x(t))\|_1
\end{equation*}
Noting that $\frac{d}{dt} \left|\tilde{x}_i(t)-x_i(t) \right| \leq \left| \frac{d}{dt} (\tilde{x}_i(t)-x_i(t)) \right|$, we get
$$
\frac{d}{dt} \|\tilde{x}(t)-x(t)\|_1 \leq \|\tilde{\lambda}(t)-\lambda(t)\|_1 + L_g \|\tilde{x}(t)-x(t)\|_1
$$
Standard solution for inhomogeneous first order differential equation then gives \eqref{eq:sensitivity-analysis-bound}. 


\begin{lemma}
\label{lem:lipschitz-constant}
For every $\alpha(t)$ and $R(t)$, \eqref{eq:global-lipschitz} holds true for
$$L_g = 2 \left( \max_{i \in \mc E} d_i'(0) - \min_{i \in \mc E}s_i'(x^{\textrm{jam}}_i) \right).$$ 
\end{lemma}
\begin{proof}
For every $\alpha(t)$, $R(t)$, $v \in \mc V$, 
\begin{align}
\sum_{j \in \mc E_v^+} \Big| \sum_{i \in \mc E_v^-} R_{ij}(\tilde{f}_{ij}-f_{ij})\Big| 
	& \leq \sum_{j \in \mc E_v^+} \!\!\left(\!R_{ij}\! \sum_{i \in \mc E_v^-} |\bar{d}_i(\tilde{x}_i)-\bar{d}_i(x_i)| + |s_j(\tilde{x}_j) - s_j(x_j)| \right) \nonumber \\
	& \leq  \sum_{i \in \mc E_v^-} \bar{d}_i'(0) |\tilde{x}_i - x_i| - \sum_{j \in \mc E_v^+} s'_j(x^{\textrm{jam}}_j)|\tilde{x}_j - x_j| \nonumber \\
\label{eq:merge-node-bound-3}
& \leq  \sum_{i \in \mc E_v^-} d_i'(0) |\tilde{x}_i - x_i| - \sum_{j \in \mc E_v^+} s'_j(x^{\textrm{jam}}_j)|\tilde{x}_j - x_j|
\end{align}
\end{proof}
Similarly, it is easy to see that,
\begin{align}
\label{eq:merge-node-bound-4}
\sum_{i \in \mc E_v^-} \Big|\sum_{j \in \mc E_v^+} R_{ij} (\tilde{f}_{ij} - f_{ij}) \Big|\leq  \sum_{i \in \mc E_v^-} d_i'(0) |\tilde{x}_i - x_i| - \sum_{j \in \mc E_v^+} s'_j(x^{\textrm{jam}}_j)|\tilde{x}_j - x_j|
\end{align}
Combining \eqref{eq:merge-node-bound-3} and \eqref{eq:merge-node-bound-4}, we get that, for all $\alpha(t)$ and $R(t)$,
\begin{align*}
		&\|g(\tilde{x},\alpha,R)-g(x,\alpha,R)\|_1 \\
		& \qquad\qquad \leq \sum_{i \in \mc E} \Big| \sum_{j \in \mc E_i^-} R_{ji}(\tilde{f}_{ji} -f_{ji}) \Big| + \sum_{i \in \mc E} \Big| \sum_{j \in \mc E_i^+} R_{ij}(\tilde{f}_{ij} -f_{ij}) \Big| 
			\\
		& \qquad\qquad = \sum_{v \in \mc V} \sum_{i \in \mc E_v^+} \Big| \sum_{j \in \mc E_v^-} R_{ji}(\tilde{f}_{ji} -f_{ji}) \Big| + \sum_{v \in \mc V} \sum_{i \in \mc E_v^-} \Big| \sum_{j \in \mc E_v^+} R_{ji}(\tilde{f}_{ij} -f_{ij}) \Big| 
			\\
		& \qquad\qquad \leq 2 \left( \max_{i \in \mc E} d_i'(0) - \min_{i \in \mc E}s_i'(x^{\textrm{jam}}_i) \right) \|\tilde{x}-x\|_1
\end{align*}
\end{document}